\pdfoutput=1
\documentclass[11pt,letterpaper]{article}

\usepackage[letterpaper,margin=1in]{geometry}
\usepackage[utf8]{inputenc}
\usepackage[T1]{fontenc}
\usepackage[dvipsnames]{xcolor}
\usepackage{amsmath,amssymb,amsthm,mathtools,mathrsfs}
\usepackage{enumitem}
\usepackage{graphicx}
\usepackage{float}
\usepackage{tikz}
\usetikzlibrary{arrows.meta,calc,decorations.pathreplacing,positioning}
\usepackage{microtype}
\microtypesetup{expansion=false}
\usepackage{url}
\usepackage[numbers]{natbib}
\usepackage{mathptmx}
\usepackage{aliascnt}
\usepackage[
  unicode=true,
  bookmarks=false,
  breaklinks=false,
  pdfborder={0 0 1},
  colorlinks=true
]{hyperref}

\newcommand\myshade{85}
\colorlet{mycitecolor}{YellowOrange}
\colorlet{myurlcolor}{Aquamarine}
\definecolor{baseblue}{HTML}{DCE8F0}
\definecolor{baseblueedge}{HTML}{64849A}
\definecolor{selectedpurple}{HTML}{79538A}
\definecolor{tailgold}{HTML}{D9AA5B}
\definecolor{softgray}{HTML}{6F747A}
\definecolor{shortteal}{HTML}{4E8991}
\definecolor{longpurple}{HTML}{85558E}
\hypersetup{
  linkcolor=selectedpurple!\myshade!black,
  citecolor=mycitecolor!\myshade!black,
  urlcolor=myurlcolor!\myshade!black
}

\newtheorem{theorem}{Theorem}[section]
\newaliascnt{lemma}{theorem}
\newtheorem{lemma}[lemma]{Lemma}
\aliascntresetthe{lemma}
\newaliascnt{proposition}{theorem}
\newtheorem{proposition}[proposition]{Proposition}
\aliascntresetthe{proposition}
\newaliascnt{corollary}{theorem}

\aliascntresetthe{corollary}
\newaliascnt{fact}{theorem}
\newtheorem{fact}[fact]{Fact}
\aliascntresetthe{fact}
\newaliascnt{conjecture}{theorem}

\aliascntresetthe{conjecture}
\theoremstyle{definition}
\newaliascnt{definition}{theorem}

\aliascntresetthe{definition}
\newaliascnt{remark}{theorem}
\newtheorem{remark}[remark]{Remark}
\aliascntresetthe{remark}

\usepackage[nameinlink,capitalize]{cleveref}
\crefname{fact}{Fact}{Facts}
\Crefname{fact}{Fact}{Facts}

\allowdisplaybreaks

\title{Improved Gradient Descent Lower Bounds Beyond Nesterov}
\author{
Yuhan Ye
\\MIT\\\texttt{yyh03@mit.edu}
\and
Kaizhao Liu
\\MIT\\\texttt{mrzt@mit.edu}
}

\date{\today}

\begin{document}
\maketitle

\begin{abstract}
We study how far gradient descent (GD) can be accelerated by predetermined stepsizes in smooth convex optimization. Going beyond the classical $\Omega(n^{-2})$ first-order oracle lower bound~\cite{NemirovskyYudin1983}, we prove an $\Omega(n^{-1.6342})$ non-anytime bound and an $\Omega(n^{-1.2408})$ anytime barrier. These improve the $\Omega(n^{-1.932})$ non-anytime result of~\cite{MaChen2026} and the $\Omega(n^{-4/3})$ anytime bound of~\cite{TsaiFatkhullinZhangHe2026}, respectively. Both bounds continue to hold when the stepsizes may be negative.
Our anytime lower bound shows that the \(O(n^{-\log_2(1+\sqrt{2})})\) convergence rate of non-anytime silver schedules~\cite{AltschulerParrilo2023,GrimmerShuWang2025Composing} is unattainable in the anytime case. This establishes a strict separation between the two settings.
\end{abstract}

\clearpage
\tableofcontents
\clearpage

\section{Introduction}
\label{sec:introduction}
\textbf{Gradient descent (GD)} is one of the simplest optimization algorithms, dating back nearly two hundred years to Cauchy~\cite{Cauchy1847}.
It updates
\[
  x_{k+1}=x_k-h_k\nabla f(x_k),
\]
where $h_k\in\mathbb R$ is fixed in advance. In this paper, we study how far this basic method can be accelerated by predetermined stepsizes in smooth convex optimization.

To formalize this question, we define the convergence rate of a stepsize schedule $H=(h_1,\ldots,h_n)$ by
\[
  R_n(H):=
  \sup_{d\in\mathbb N}
  \sup_{f\in\mathcal F_L(\mathbb R^d)}
  \sup_{x^\star\in\arg\min f}
  \sup_{x_1\in\mathbb R^d\setminus\{x^\star\}}
  \frac{f(x_{n+1})-f(x^\star)}
  {\frac{L}{2}\lVert x_1-x^\star\rVert^2},
\]
where $\mathcal F_L(\mathbb R^d)$ is the class of convex $L$-smooth functions on $\mathbb R^d$ with a nonempty set of minimizers. In the non-anytime (finite-horizon) setting, the schedule is designed for a fixed horizon $n$. In the anytime setting, one infinite schedule $h=(h_k)_{k\ge1}$ is used for every horizon, with $H_n=(h_1,\ldots,h_n)$.

It is a standard textbook result that GD with the constant stepsize $1/L$ achieves an $O(n^{-1})$ rate~\cite{LevitinPolyak1966,Nesterov2004}.
Classical acceleration approaches modify the GD iteration by introducing momentum or auxiliary sequences, as in Polyak's heavy-ball method~\cite{Polyak1964} and Nesterov's accelerated method~\cite{Nesterov1983}. 
Among the broader class of first-order methods, Nesterov's method attains the optimal $O(n^{-2})$ rate for smooth convex objectives, while the matching $\Omega(n^{-2})$ first-order oracle lower bound goes back to Nemirovsky and Yudin~\cite{NemirovskyYudin1983}.

For GD with predetermined stepsizes, since it is a restricted class of first-order methods, the classical $\Omega(n^{-2})$ lower bound continues to apply. For many years, however, the best known general upper bound remained the classical $O(n^{-1})$ rate. This leads to the fundamental question of whether a faster rate is possible. In recent years, a growing body of work surprisingly shows that GD itself can be accelerated beyond the classical $O(n^{-1})$ rate by using carefully designed stepsize schedules that use occasional long steps and recursive structure~\cite{Grimmer2024,AltschulerParrilo2023,GrimmerShuWang2025LongSteps,GrimmerShuWang2025Composing}.
The current best known non-anytime upper bound is $O(n^{-\log_2(1+\sqrt{2})})$, achieved by the silver schedule~\cite{AltschulerParrilo2023,GrimmerShuWang2025Composing}.
An anytime construction based on silver-schedule blocks achieves an $O(n^{-1.119})$ rate at every stopping time~\cite{ZhangLeeDuChen2025}.
This left open whether the generic $\Omega(n^{-2})$ lower bound could be strengthened  toward these upper bounds.

Recently, Ma and Chen~\cite{MaChen2026} proved an $\Omega(n^{-1.932})$ lower bound by cleverly constructing a hard function family. A subsequent refinement~\cite{Tsai2026} gave an exposition of the framework and improved this bound to $\Omega(n^{-\sqrt{3}})$. For the anytime case, \cite{TsaiFatkhullinZhangHe2026} provides an $\Omega(n^{-4/3})$ lower bound.\footnote{Throughout, we use an $\Omega(n^{-p})$ anytime lower bound as shorthand for the statement that no single infinite stepsize schedule achieves an $o(n^{-p})$ rate.}

\subsection{Contribution}
Despite the recent progress, substantial gaps still remain between the best known lower and upper bounds in both the non-anytime and anytime settings. In this paper, we narrow these gaps by improving the non-anytime lower bound to $\Omega(n^{-1.6342})$ and the anytime lower bound to $\Omega(n^{-1.2408})$.

\begin{theorem}
\label{thm:fixed-horizon}
There exists a constant $c>0$ such that, for every integer $n\ge1$ and every $H\in\mathbb R^n$,
\[
  R_n(H)\ge c n^{-1.6342}.
\]
\end{theorem}

\begin{theorem}
\label{thm:anytime}
No infinite schedule $h=(h_k)_{k\ge1}\in\mathbb R^{\mathbb N}$, with $H_n=(h_1,\ldots,h_n)$, satisfies
\[
  R_n(H_n)=o(n^{-1.2408}).
\]
\end{theorem}

Two aspects of our results are worth emphasizing.
\begin{enumerate}[label=\textup{(\roman*)},leftmargin=*]
  \item Our anytime lower bound shows that the non-anytime silver rate $O(n^{-\log_2(1+\sqrt{2})})$~\cite{AltschulerParrilo2023,GrimmerShuWang2025Composing} is not achievable in the anytime case. This establishes a strict separation between the non-anytime and anytime settings.
  \item Both lower bounds remain valid when negative stepsizes are allowed, resolving the corresponding question left open in~\cite{MaChen2026,TsaiFatkhullinZhangHe2026}.
\end{enumerate}

\begin{figure}[H]
  \centering
  \includegraphics[width=\textwidth]{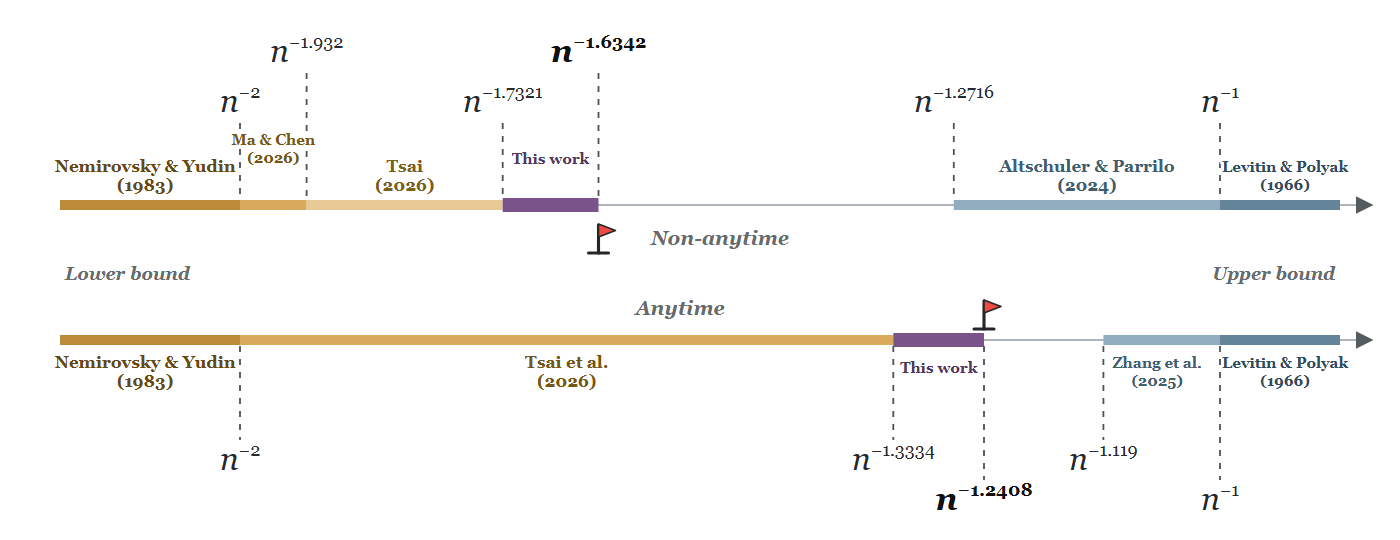}
  \caption{Improved lower bounds for GD.}
  \label{fig:rate-landscape}
\end{figure}

After reviewing additional related work in \cref{sec:related-work}, we give an overview of our techniques in \cref{sec:technical-overview-fixed}. We first consider the case of nonnegative stepsize schedules. At a high level, our proofs use the hard-function family introduced in~\cite[Theorem~4.1]{MaChen2026}.
After this construction, Ma and Chen summarize the selected long-step excesses by a single quantity~\cite[Section~5.4]{MaChen2026}. Tsai rewrites this quantity in terms of the harmonic mean of these excesses~\cite[Lemma~7]{Tsai2026}. In contrast, our main technical contribution is to retain every factor coupling consecutive selected long steps and estimate these factors term by term.
For the anytime case, we adapt the finite-to-anytime transfer from~\cite[Theorem~4.1]{TsaiFatkhullinZhangHe2026} and apply the same estimate at suitably chosen horizons. This yields the improved anytime lower bound.

The remainder of the paper is organized as follows.
\Cref{sec:nonanytime-proof} provides the full details of the non-anytime proof outlined in \cref{sec:technical-overview-fixed}.
\Cref{sec:anytime} extends the analysis to the anytime setting.
\Cref{app:hard-functions} recalls the hard-function construction, \cref{app:scalar-variational-problem} details how the scalar parameters are chosen, and \cref{app:anytime-proofs} contains the proofs deferred from \cref{sec:anytime}. Finally, \cref{app:negative-stepsizes} explains how we extend both lower bounds to schedules whose stepsizes may be negative.

\subsection{Related Work}
\label{sec:related-work}

\paragraph{Accelerating GD.}
Classical acceleration methods modify GD by incorporating information from previous iterations, as in Polyak's heavy-ball method~\cite{Polyak1964} and Nesterov's accelerated method~\cite{Nesterov1983}.
A complementary line of work uses the performance-estimation problem (PEP), which provides a systematic numerical framework for worst-case analysis~\cite{DroriTeboulle2014}. Interpolation conditions yield exact PEP formulations for fixed-step first-order methods~\cite{TaylorHendrickxGlineur2017}.
In the stepsize-only setting, PEP has also been used to search numerically for horizon-dependent GD schedules, providing finite-horizon evidence of acceleration~\cite{DasGuptaVanParysRyu2024,KamriHendrickxGlineur2026}. 
Within the classical $O(n^{-1})$ regime, improved constants were obtained using predetermined schedules with stepsizes increasing toward $2/L$~\cite{TeboulleVaisbourd2023} and, separately, periodic schedules with occasional long steps analyzed through multi-step certificates~\cite{Grimmer2024}.
For strongly convex quadratic objectives, stepsize-only acceleration was known much earlier through Chebyshev stepsizes~\cite{Young1953}. Fractal orderings were later introduced to control their unstable intermediate iterates~\cite{AgarwalGoelZhang2021}.

\paragraph{Silver and recursive stepsize schedules.}
Let $\rho=1+\sqrt{2}$. At horizons $n=2^k-1$, the recursively defined silver schedule attains the rate $O(n^{-\log_2\rho})$, where $\log_2\rho\approx1.2716$~\cite{AltschulerParrilo2023}.
A related right-heavy schedule attains the silver exponent for the objective gap, while its reversal attains the same exponent for the squared gradient norm~\cite{GrimmerShuWang2025LongSteps}. A subsequent composition framework and an independent concatenation construction extend this rate to every prescribed finite horizon~\cite{GrimmerShuWang2025Composing,ZhangJiang2026}.
Note that these guarantees hold at selected horizons (non-anytime case). The question of whether an infinite schedule can accelerate GD at every stopping time was posed in~\cite{KornowskiShamir2024} and answered affirmatively in~\cite{ZhangLeeDuChen2025}, which gives an anytime $O(n^{-1.119})$ rate.
The silver-step constructions also extend to projected and proximal gradient methods~\cite{BokAltschuler2025}, and to smooth strongly convex objectives with accelerated linear rates~\cite{AltschulerParrilo2025HedgingI}.

\paragraph{Lower bounds for GD.}
The classical $\Omega(n^{-2})$ first-order oracle lower bound for smooth convex optimization goes back to Nemirovsky and Yudin~\cite{NemirovskyYudin1983}, while Nesterov's accelerated method attains the matching $O(n^{-2})$ rate~\cite{Nesterov1983}. A standard quadratic-chain proof appears in~\cite[Section~2.1.2, Theorem~2.1.7]{Nesterov2004}, and we recall its geometric mechanism in \cref{fact:quadratic-chain}.
Stronger conclusions were previously known under additional structural restrictions. Time-invariant oblivious first-order methods cannot attain $O(n^{-\alpha})$ rates for any $\alpha>1$~\cite[Corollary~1]{ArjevaniShamir2016}. Within the recursively generated class of basic $f$-composable GD schedules, the best objective-gap rate is $\Theta(n^{-\log_2(1+\sqrt{2})})$~\cite[Theorem~5]{GrimmerShuWang2025Composing}.

For arbitrary predetermined nonnegative schedules, an $\Omega(n^{-1.932})$ lower bound was proved using a hard function adapted to the schedule's long steps~\cite{MaChen2026}. The same construction was then used to sharpen the non-anytime bound to $\Omega(n^{-\sqrt{3}})$~\cite{Tsai2026}.
For the anytime case, it was proved in~\cite{TsaiFatkhullinZhangHe2026} that no positive predetermined infinite schedule satisfies $R_n(H_n)=o(n^{-4/3})$.

\paragraph{The role of negative stepsizes.}
Surprisingly, it was shown in~\cite{ShugartAltschuler2025} that suitably designed stepsize schedules that periodically take negative values can make  gradient descent--ascent methods (GDA) converge on convex--concave problems for which standard GDA fails.
For smooth convex minimization, the lower bounds in~\cite{MaChen2026,TsaiFatkhullinZhangHe2026} were proved only for nonnegative schedules. Thus, it remains open whether negative stepsizes can improve the best achievable rate beyond what is possible with nonnegative schedules.

\section{Technique overview}
\label{sec:technical-overview-fixed}

In this section, we provide a high-level overview of techniques behind the proof of \cref{thm:fixed-horizon}. Replacing $f$ by $f/L$ and each $h_k$ by $Lh_k$ leaves the GD iterates and $R_n$ unchanged, so we assume $L=1$ throughout. We begin with nonnegative schedules $H\in[0,\infty)^n$. Following~\cite{MaChen2026}, we call a step $h_k$ long if $h_k>1$ and short if $h_k\leq 1$.
Let 
\[
  r:=\#\{k:h_k>1\}
\]
be the number of long steps.
For an integer $1\le q\le r$, select $q$ long steps
\[
  0<t_1<\cdots<t_q\le n,
  \qquad h_{t_i}>1,
\]
where $t_0=0$ and $t_{q+1}=n+1$. 
Define $S_i$ to be the stepsize accumulation between the selected long steps, that is,
\[
  S_i:=h_{t_{i-1}+1:t_i-1},
  \qquad 1\le i\le q+1,
\]
where $h_{a:b}:=\sum_{k=a}^b h_k$, with an empty sum equal to zero.  
An illustration of these concepts is provided in \cref{fig:hS}.

\begin{figure}[htbp]
  \centering
  \begin{tikzpicture}[font=\small]
  \begin{scope}[x=.58cm,y=1.05cm]
    \input{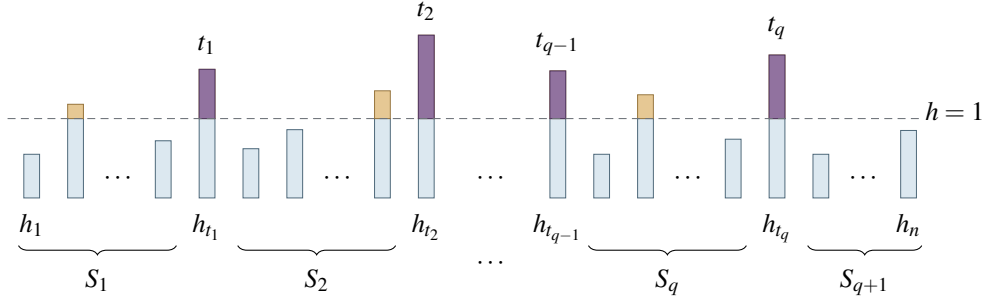}
  \end{scope}
\end{tikzpicture}
  \caption{The selected long steps are $h_{t_i}>1$, and each brace marks the stepsize accumulation $S_i$ between the selected long steps.}
  \label{fig:hS}
\end{figure}

\begin{lemma}
\label{lem:hard-product}
For every such selection,
\begin{equation}
\label{eq:hard-product}
  R_n(H)\ge
  \frac{1}{S_q+h_{t_q}+2S_{q+1}+1}
  \prod_{i=1}^{q-1}
  \frac{S_{i+1}+h_{t_{i+1}}}
  {S_i+h_{t_i}+S_{i+1}+h_{t_{i+1}}}
  \prod_{i=1}^{q}\frac{h_{t_i}-1}{S_i+1}.
\end{equation}
\end{lemma}

This is the lower bound in~\cite[Theorem~4.1]{MaChen2026}; see also~\cite[Lemma~6]{Tsai2026}.  
The lemma reduces the original problem to establishing a lower bound for the right-hand side of \eqref{eq:hard-product}, which involves only the sequence $H$. 
When no long step is selected, the construction based on the classical Huber loss gives
\begin{equation}\label{eq:Huber}
  R_n(H)\ge \frac{1}{1+2\sum_{k=1}^n h_k}.
\end{equation}
We briefly recall the proof of \eqref{eq:Huber} in \cref{huberfact}.
A nonempty selection $q>0$ allows \eqref{eq:hard-product} to adapt flexibly to the long steps, limiting the acceleration they might otherwise provide in \eqref{eq:Huber}.
The lemma is proved using a hard function tailored to the schedule $H$ and the selected long steps. \Cref{app:hard-functions} gives an intuitive explanation of the construction.

Based on \cref{lem:hard-product}, Ma and Chen bound the resulting product using a single quantity that summarizes the selected long-step excesses~\cite[Section~5.4]{MaChen2026}. Tsai derives a related bound in terms of the harmonic mean of these excesses~\cite[Lemma~7]{Tsai2026} and then uses a counting function~\cite[Lemma~8]{Tsai2026}. We instead retain each factor coupling two consecutive selected long steps and estimate these factors term by term.

\paragraph{Reduction to a sequence inequality.}

For $w,z,x,y>0$, define
\begin{equation}
\label{eq:kernel}
  K_{w,z}(x,y):=
  \frac{\sqrt{xy}\,(w+z+x+y)}
  {\sqrt{wz(w+x)(z+y)}}.
\end{equation}
Given any sequence $\omega_1,\ldots,\omega_q$, write $\omega_1^\downarrow\ge\cdots\ge\omega_q^\downarrow$ for its decreasing rearrangement.
Then the problem of proving an $\Omega(n^{-(1+\alpha)})$ lower bound for \eqref{eq:hard-product} can be reduced to the following sequence inequality.

\begin{lemma}
\label{lem:fixed-sequence-reduction}
Fix $\alpha>0$.  
Suppose there is a constant $C_\alpha<\infty$, depending only on $\alpha$, with the following property:  
For every integer $q\geq 2$ and any positive sequences $x_1,\ldots,x_q$ and $\omega_1,\ldots,\omega_q$ such that 
\begin{equation}\label{eq:sequence-constraints}
  \sum_{i=1}^q x_i<q,
  \qquad
  \sum_{s=p+1}^q\omega_s^\downarrow
  \ge q\left[\left(\frac qp\right)^\alpha-1\right]
  \quad (\forall 1\le p<q),
\end{equation}
we have
\begin{equation}
\label{eq:sequence-bound}
  E_{\mathrm{end}}
  \prod_{i=1}^{q-1}
  K_{\omega_i,\omega_{i+1}}(x_i,x_{i+1})
  \le C_\alpha,
\end{equation}
where
\begin{equation}
\label{eq:endpoint-factor}
  E_{\mathrm{end}}:=
  \sqrt{x_1x_q}
  \sqrt{\frac{1+x_1/\omega_1}{1+x_q/\omega_q}}
  \left[
    1+\frac{x_q+2\left(q-\sum_{i=1}^q x_i\right)}{\omega_q}
  \right].
\end{equation}
Then, $R_n(H)\ge c_\alpha(n+1)^{-(1+\alpha)}$ for every $n$ and every $H\in[0,\infty)^n$.
\end{lemma}

The reduction makes two choices that are not apparent from \eqref{eq:hard-product}. Given a cutoff $q$, it selects the $q$ largest excesses, while the cutoff itself is chosen from the schedule $H$ according to \eqref{eq:choice-of-q}. The right-hand side of \eqref{eq:hard-product} is then normalized so that the variables $x_i$ have total mass below $q$, and the decreasing rearrangement of the weights $\omega_i$ satisfies the tail inequalities in \eqref{eq:sequence-constraints}.
The proof is provided in \cref{sec:fixed-sequence-reduction-proof}.

\paragraph{Eliminating the variables \texorpdfstring{$x_i$}{xi}.}

To make \eqref{eq:sequence-bound} tractable, we first eliminate the variables $x_i$.  
For a parameter $\lambda>0$ to be chosen later, define
\begin{equation}
\label{eq:gamma-definition}
  \Gamma_\lambda(w,z):=
  \sup_{x,y>0}
  \left\{
    \log K_{w,z}(x,y)-\lambda(x+y)
  \right\}.
\end{equation}
By definition,
\begin{equation}
\label{eq:gamma-envelope}
  \log K_{w,z}(x,y)
  \le \lambda(x+y)+\Gamma_\lambda(w,z).
\end{equation}
Applying this inequality to every adjacent factor and using $\sum_i x_i<q$ gives the next reduction.

\begin{lemma}
\label{lem:envelope-reduction}
Fix $\alpha,\lambda>0$. 
There is a constant $C_{\alpha,\lambda}<\infty$ such that, for every integer $q\ge2$, every pair of positive sequences $(x_i)_{i=1}^q$ and $(\omega_i)_{i=1}^q$ satisfying \eqref{eq:sequence-constraints} obeys
\begin{equation}
\label{eq:envelope-reduction}
  E_{\mathrm{end}}
  \prod_{i=1}^{q-1}K_{\omega_i,\omega_{i+1}}(x_i,x_{i+1})
  \le C_{\alpha,\lambda}
  \exp\!\left\{
    2\lambda q+
    \sum_{i=1}^{q-1}\Gamma_\lambda(\omega_i,\omega_{i+1})
  \right\}.
\end{equation}
\end{lemma}

We are left with a sum of consecutive $\Gamma_\lambda$ terms.  
To proceed, we record some basic properties of $\Gamma_\lambda$.

\begin{lemma}[Properties of $\Gamma_\lambda$]
\label{lem:gamma-geometry}
For every $\lambda>0$, $\Gamma_\lambda$ has the following properties.
\begin{enumerate}
  \item \emph{Smoothness.} For every $w,z>0$, the supremum in
  \eqref{eq:gamma-definition} is attained at a unique point of
  $(0,\infty)^2$. Denote this point by
  $(x_\lambda(w,z),y_\lambda(w,z))$. The optimizer map and
  $\Gamma_\lambda$ are $C^\infty$ on $(0,\infty)^2$.
  \item \emph{Symmetry.}
  $\Gamma_\lambda(w,z)=\Gamma_\lambda(z,w)$.
  \item \emph{Joint convexity.} The map
  $(w,z)\mapsto\Gamma_\lambda(w,z)$ is jointly convex on
  $(0,\infty)^2$.
  \item \emph{Coordinatewise decrease.} The function $\Gamma_\lambda$ is
  strictly decreasing in each coordinate.
  \item \emph{Strict submodularity.}
  \begin{equation}
  \label{eq:submodularity}
    \partial_{wz}^2\Gamma_\lambda(w,z)<0.
  \end{equation}
\end{enumerate}
\end{lemma}

\paragraph{Controlling the summation of consecutive pairs.}

The remaining sum is bounded in four steps. 
First, the summation of consecutive pairs can be split into two matchings, and maximizing over matchings leads to a symmetric function.
Second, the symmetry, joint convexity, and coordinatewise decrease together allow majorization, replacing the unknown weights $\omega_i^\downarrow$ by the explicit comparison sequence in \eqref{eq:reference-weights}.  
Third, strict submodularity identifies the maximizing matching of these explicit comparison weights.  
Finally, the resulting finite sum is compared with a one-dimensional integral.  
The full argument is given in \cref{sec:path-estimate-proof}.

\begin{lemma}
\label{lem:path-estimate}
Fix $\alpha,\lambda>0$. There is a constant $C_{\alpha,\lambda}<\infty$ such that, for every integer $q\ge2$ and every positive sequence $\omega_1,\ldots,\omega_q$ satisfying the tail condition in \eqref{eq:sequence-constraints},
\begin{equation}
\label{eq:path-estimate}
  \sum_{i=1}^{q-1}\Gamma_\lambda(\omega_i,\omega_{i+1})
  \le
  2q\int_0^{1/2}
  \Gamma_\lambda\bigl(W_\alpha(t),W_\alpha(1-t)\bigr)\,dt
  +C_{\alpha,\lambda},
\end{equation}
where $W_\alpha(t):=\alpha t^{-1-\alpha}$ for $0<t\le1$.
\end{lemma}

\paragraph{Completing the proof of
\texorpdfstring{\Cref{thm:fixed-horizon}.}{Theorem 1.1.}}
\label{sec:fixed-proof-overview-completion}
Define
\begin{equation}
\label{eq:J-definition}
  J(\alpha,\lambda)
  :=2\lambda+
  2\int_0^{1/2}
  \Gamma_\lambda\bigl(W_\alpha(t),W_\alpha(1-t)\bigr)\,dt.
\end{equation}
Then \cref{lem:envelope-reduction,lem:path-estimate} give
\begin{equation}
\label{eq:product-by-J}
  E_{\mathrm{end}}
  \prod_{i=1}^{q-1}K_{\omega_i,\omega_{i+1}}(x_i,x_{i+1})
  \le C_{\alpha,\lambda}e^{qJ(\alpha,\lambda)}.
\end{equation}
Thus $J(\alpha,\lambda)\le0$ makes \eqref{eq:sequence-bound} uniform in $q$, and \cref{lem:fixed-sequence-reduction} yields
\begin{equation}
\label{eq:fixed-horizon-rate}
  R_n(H)\ge c_{\alpha,\lambda}(n+1)^{-(1+\alpha)}.
\end{equation}
For $\alpha=0.6342$ and $\lambda=0.4506$, numerical integration gives $J(0.6342,0.4506)<-0.00005<0$.
\Cref{app:scalar-variational-problem} studies the dependence of $J$ on $(\alpha,\lambda)$ and explains how the pair $(0.6342,0.4506)$ was obtained numerically.
Therefore $R_n(H)=\Omega(n^{-1.6342})$, proving the nonnegative case of \cref{thm:fixed-horizon}. \Cref{app:negative-stepsizes} further extends the result to arbitrary stepsize schedules.

\section{Proof of the Non-Anytime Lower Bound}
\label{sec:nonanytime-proof}

This section proves \cref{lem:fixed-sequence-reduction,lem:envelope-reduction,lem:gamma-geometry,lem:path-estimate} and thereby completes the proof of \cref{thm:fixed-horizon} for nonnegative stepsize schedules.

\subsection{Reduction to a sequence inequality}
\label{sec:fixed-sequence-reduction-proof}

\begin{proof}[Proof of \cref{lem:fixed-sequence-reduction}]
Suppose that \eqref{eq:sequence-constraints} implies \eqref{eq:sequence-bound}, with a constant $C_\alpha$ uniform over all $q\ge 2$.  
Fix $n\ge 1$ and a nonnegative schedule $H=(h_1,\ldots,h_n)$.
Recall that $r$ is the number of long steps.
If $r=0$, the Huber bound \eqref{eq:Huber} yields 
\[
  R_n(H)\ge(2n+1)^{-1}\ge(2(n+1))^{-1}.
\]
Therefore, from now on we consider $r\geq 1$.

We decompose each stepsize $h_k$ into two components: a base $\min\{h_k, 1\}$ and an excess $\max\{h_k - 1, 0\}$.
Let \(a_1 \ge \cdots \ge a_r > 0\) denote the positive excesses in decreasing order.
We define the base accumulation as $B:=1+\sum_{k=1}^n\min\{h_k,1\}$.
By definition, $B\leq n+1$.
Moreover, let $D_s$ denote the total stepsize accumulation excluding the excesses from the $s$ longest steps:
\[
  D_s:=B+\sum_{\ell=s+1}^r a_\ell,
  \qquad 1\le s\le r,
\]
with $D_r=B$.
See \cref{fig:ranked-excesses} for an illustration of these definitions.

\begin{figure}[H]
  \centering
  \begin{minipage}[t]{.55\textwidth}
  \vspace{0pt}
  \centering
  \begin{tikzpicture}[font=\scriptsize]
    \begin{scope}[x=.37cm,y=.76cm]
      \path[use as bounding box] (-.70,-1.15) rectangle (21.60,2.30);
      \input{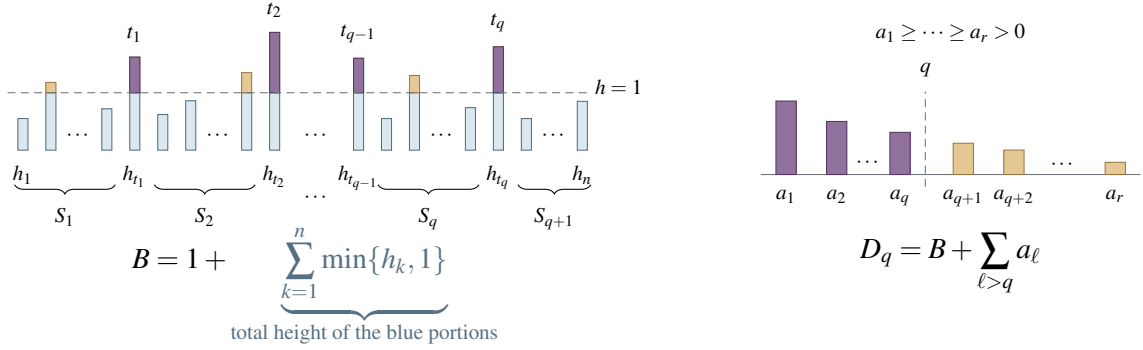}
    \end{scope}
  \end{tikzpicture}

  \vspace{.2em}
  $\displaystyle
    B=1+\textcolor{baseblueedge!85!black}{
      \underbrace{\sum_{k=1}^n\min\{h_k,1\}}
      _{\text{\scriptsize total height of the blue portions}}}$
\end{minipage}\hfill
\begin{minipage}[t]{.42\textwidth}
  \vspace{0pt}
  \centering
  \begin{tikzpicture}[font=\scriptsize,x=.67cm,y=.90cm]
    \path[use as bounding box] (.50,-.61) rectangle (8.50,2.30);
    \node at (4.5,2.08) {$a_1\ge\cdots\ge a_r>0$};

    \draw[softgray] (.75,0)--(8.25,0);
    \foreach \x/\h/\lab in {
      1.25/1.08/a_1,2.25/.78/a_2,3.50/.62/a_q}{
      \fill[selectedpurple!82] (\x-.20,0) rectangle (\x+.20,\h);
      \draw[selectedpurple!70!black,line width=.4pt]
        (\x-.20,0) rectangle (\x+.20,\h);
      \node[below=2pt] at (\x,0) {$\lab$};
    }
    \node at (2.88,.18) {$\cdots$};
    \foreach \x/\h/\lab in {
      4.75/.46/a_{q+1},5.75/.36/a_{q+2},7.75/.18/a_r}{
      \fill[tailgold!65] (\x-.20,0) rectangle (\x+.20,\h);
      \draw[tailgold!70!black,line width=.35pt]
        (\x-.20,0) rectangle (\x+.20,\h);
      \node[below=2pt] at (\x,0) {$\lab$};
    }
    \node at (6.72,.15) {$\cdots$};
    \draw[densely dashed,softgray] (4.00,-.18)--(4.00,1.22)
      node[above=2pt,black] {$q$};
  \end{tikzpicture}

  \vspace{.2em}
  $\displaystyle D_q=B+\sum_{\ell>q}a_\ell$
\end{minipage}
  \caption{The schedule in chronological order
  (left) and its positive excesses in decreasing order (right).}
  \label{fig:ranked-excesses}
\end{figure}

Choose
\begin{equation}\label{eq:choice-of-q}
    q\in\operatorname*{arg\,min}_{1\le s\le r}D_s s^\alpha.
\end{equation}
This specific choice ensures that $D_q$ has a universal upper bound
\begin{equation}\label{eq:Dq-universal-upper-bound}
    D_q\leq D_q q^\alpha \le D_r r^\alpha \leq (n+1)^{\alpha+1}
\end{equation}
as $D_r=B\leq n+1$ and $r+1\leq n+1$.

We first treat the special case $q=1$ to illustrate the power of \cref{lem:hard-product}.
Then we show that \cref{lem:hard-product} combined with \eqref{eq:sequence-bound} is sufficient to prove the general case $q\geq 2$.

\paragraph{Case 1: $q=1$.}
If $a_1\le D_1$, since $D_1+a_1=1+\sum_{k=1}^n h_k$, the Huber bound \eqref{eq:Huber} gives
\[
  R_n(H)\ge \frac1{2(D_1+a_1)-1}
  \ge \frac1{4D_1}.
\] 
If $a_1>D_1$, select $t$ such that $h_t-1=a_1$ and apply \cref{lem:hard-product}.  
In this case $D_1=(S_1+1)+(S_2+1)$, while
\[
  (S_1+1)+a_1+2(S_2+1)-1<a_1+2D_1<3a_1.
\]
Consequently,
\[
  R_n(H)
  \ge\frac{a_1}{(S_1+1)(S_1+a_1+2S_2+2)}
  >\frac1{3(S_1+1)}
  \ge\frac1{3D_1}
  >\frac1{4D_1}.
\]
Thus $R_n(H)\ge1/(4D_1)$ in both subcases.  Applying \eqref{eq:Dq-universal-upper-bound} yields
\[
  R_n(H)\ge\frac1{4(n+1)^{\alpha+1}}.
\]

\paragraph{Case 2: $q\ge 2$.}
Let $t_1<\cdots<t_q$ be the indices of the $q$ largest positive excesses $h_k-1$.  
With this selection, let $S_1,\ldots,S_{q+1}$ be defined as in \cref{lem:hard-product}.  
Note that
\[
  \sum_{i=1}^{q+1}(S_i+1)
  =q+1+\sum_{k\notin\{t_1,\ldots,t_q\}}h_k
  =D_q.
\]
We use $D_q/q$ as the normalization scale for $S_i+1$ and $h_{t_i}-1$, and define
\[
  x_i:=\frac{q(S_i+1)}{D_q},
  \qquad
  \omega_i:=\frac{q(h_{t_i}-1)}{D_q},
  \qquad 1\le i\le q.
\]
By \cref{lem:hard-product}, $R_n(H)\ge G$, where $G$ denotes the right-hand side of \eqref{eq:hard-product} for the indices selected above.
Substituting the definitions of $x_i$ and $\omega_i$ gives
\[
  G^{-1}
  =\frac{D_q}{q}\sqrt{x_1x_q}
  \sqrt{\frac{1+x_1/\omega_1}{1+x_q/\omega_q}}
  \left[
    1+\frac{x_q+2(q-\sum_{i=1}^q x_i)-q/D_q}{\omega_q}
  \right] 
  \prod_{i=1}^{q-1}
  K_{\omega_i,\omega_{i+1}}(x_i,x_{i+1}).
\]
Replacing $-q/D_q$ by $0$ increases the quantity in the bracket and yields 
\begin{equation}\label{eq:top-q-lower-bound}
    G^{-1}\le \frac{D_q}{q}E_{\mathrm{end}}\prod_{i=1}^{q-1}K_{\omega_i,\omega_{i+1}}(x_i,x_{i+1}).
\end{equation}
The sequences $(x_i)_{i=1}^q$ and $(\omega_i)_{i=1}^q$ satisfy \eqref{eq:sequence-constraints}. Specifically, as $x_i$ are normalized,
\[
  \sum_{i=1}^q x_i
  =\frac q{D_q}\sum_{i=1}^q(S_i+1)
  =q-\frac{q(S_{q+1}+1)}{D_q}
  <q.
\]
Moreover, noting that $\omega_s^\downarrow=qa_s/D_q$ for $1\le s\le q$, the minimizing property \eqref{eq:choice-of-q} implies
\[
  1+\frac1q\sum_{s=p+1}^q\omega_s^\downarrow
  =1+\frac{\sum_{s=p+1}^q a_s}{D_q}
  =\frac{D_p}{D_q}
  \ge \left(\frac qp\right)^\alpha
\]
for every $1\le p<q$.
Hence, \eqref{eq:sequence-bound} holds, and by \eqref{eq:Dq-universal-upper-bound}, 
\[
  R_n(H)\ge G\ge\frac{q}{C_\alpha D_q}\geq\frac{1}{C_\alpha (n+1)^{\alpha+1}}.
\]

Finally, combining all cases gives
\[
  R_n(H)\ge c_\alpha(n+1)^{-(1+\alpha)},
  \qquad
  c_\alpha:=\min\{C_\alpha^{-1},1/4\}.
\]
\end{proof}

\subsection{Eliminating the variables \texorpdfstring{$x_i$}{xi}}
\begin{proof}[Proof of \cref{lem:envelope-reduction}]
Set
\[
  \Delta:=q-\sum_{i=1}^q x_i>0.
\]
Applying \eqref{eq:gamma-envelope} to $K_{\omega_i,\omega_{i+1}}(x_i,x_{i+1})$ for $1\le i<q$ and summing gives
\[
  \log\prod_{i=1}^{q-1}
  K_{\omega_i,\omega_{i+1}}(x_i,x_{i+1})
  \le
  \lambda\sum_{i=1}^{q-1}(x_i+x_{i+1})
  +\sum_{i=1}^{q-1}\Gamma_\lambda(\omega_i,\omega_{i+1}).
\]
Since
\[
  \sum_{i=1}^{q-1}(x_i+x_{i+1})
  =2\sum_{i=1}^q x_i-x_1-x_q
  =2q-(2\Delta+x_1+x_q),
\]
we obtain
\[
\begin{aligned}
  E_{\mathrm{end}}
  \prod_{i=1}^{q-1}K_{\omega_i,\omega_{i+1}}(x_i,x_{i+1})
  &\le
  \exp\!\left\{
    2\lambda q+\sum_{i=1}^{q-1}
    \Gamma_\lambda(\omega_i,\omega_{i+1})
  \right\} E_{\mathrm{end}}
  e^{-\lambda(2\Delta+x_1+x_q)}.
\end{aligned}
\]
It remains to bound $E_{\mathrm{end}}e^{-\lambda(2\Delta+x_1+x_q)}$ independently of $q$, $(x_i)_{i=1}^q$, and $(\omega_i)_{i=1}^q$.

Taking $p=q-1$ in \eqref{eq:sequence-constraints} yields
\begin{equation}\label{eq:weight-lower-bound}
    \min_{1\le i\le q}\omega_i
  =\omega_q^\downarrow
  \ge q\left[\left(\frac q{q-1}\right)^\alpha-1\right]
  \ge\alpha q\log\frac q{q-1}
  \ge\alpha.
\end{equation}
The last two inequalities use $e^u-1\ge u$ and $\log(1+t)\ge t/(1+t)$.  In particular, $\omega_1,\omega_q\ge\alpha$.  
Using this in \eqref{eq:endpoint-factor} and writing $T:=x_1+x_q+\Delta$, we obtain
\[
  E_{\mathrm{end}}e^{-\lambda(2\Delta+x_1+x_q)}
  \le C_\alpha(1+T)^{5/2}e^{-\lambda T}
  \le C_\alpha\sup_{t\ge0}(1+t)^{5/2}e^{-\lambda t}
  =:C_{\alpha,\lambda}<\infty.
\]
Substitution into the preceding display proves \eqref{eq:envelope-reduction}.
\end{proof}

\subsection{Properties of \texorpdfstring{$\Gamma_\lambda$}{Gamma lambda}}
\begin{proof}[Proof of \cref{lem:gamma-geometry}]
For $w,z,x,y>0$, set
\[
  F_{w,z}(x,y):=\log K_{w,z}(x,y)-\lambda(x+y).
\]

\paragraph{(1) Smoothness.}
For fixed $w,z>0$, the function $F_{w,z}(x,y)$ tends to $-\infty$ as $x\downarrow0$, $y\downarrow0$, or $x+y\to\infty$.  
It therefore attains its maximum at an interior point.  
Write $A:=w+x$, $B:=z+y$, and $\Sigma:=A+B$.  
For every nonzero direction $(a,b)$,
\[
  D^2_{x,y}F_{w,z}(x,y)[(a,b),(a,b)]
  =-\frac{a^2}{2x^2}-\frac{b^2}{2y^2}
   -\frac{(a+b)^2}{\Sigma^2}
   +\frac{a^2}{2A^2}+\frac{b^2}{2B^2}<0.
\]
Thus the maximizer is unique.

Fix $(w_0,z_0)$ and let $(x_0,y_0)$ be its maximizer.  
The Jacobian in $(x,y)$ of the first-order condition $\nabla_{x,y}F_{w,z}(x,y)=0$ is the invertible Hessian above.  
The implicit function theorem gives a neighborhood of $(w_0,z_0)$ on which the critical point is a $C^\infty$ function of $(w,z)$.  
Since $(w_0,z_0)$ was arbitrary, the optimizer map is $C^\infty$ on $(0,\infty)^2$.  
Substituting it into $F_{w,z}$ shows that $\Gamma_\lambda$ is also $C^\infty$.

\paragraph{(2) Symmetry.}
The definition \eqref{eq:gamma-definition} is invariant under
$(w,x)\leftrightarrow(z,y)$.

\paragraph{(3) Joint convexity.}
For fixed $x,y$, put $A:=w+x$ and $B:=z+y$.  The Hessian in $(w,z)$ has quadratic form
\[
  D^2_{w,z}F_{w,z}(x,y)[(a,b),(a,b)]=\frac{a^2}{2w^2}+\frac{a^2}{2A^2}
  +\frac{b^2}{2z^2}+\frac{b^2}{2B^2}
  -\frac{(a+b)^2}{(A+B)^2}\ge0.
\]
Indeed, the first four terms dominate $a^2/A^2+b^2/B^2$, which is at least $(a+b)^2/(A+B)^2$ by Cauchy--Schwarz.  
Taking the supremum in \eqref{eq:gamma-definition} preserves convexity.

\paragraph{(4) Coordinatewise decrease.}
Write $x_*:=x_\lambda(w,z)$ and $y_*:=y_\lambda(w,z)$, and set $A:=w+x_*$, $B:=z+y_*$, and $\Sigma:=A+B$.  
The chain rule and the first-order conditions give
\begin{equation}\label{eq:partialwGammalambda}
    \begin{aligned}
  \partial_w\Gamma_\lambda
  &=\partial_wF_{w,z}
    +\partial_xF_{w,z}\,\partial_wx_*
    +\partial_yF_{w,z}\,\partial_wy_*\\
  &=\partial_wF_{w,z}
  =\frac1\Sigma-\frac1{2w}-\frac1{2A}.
\end{aligned}
\end{equation}
Equivalently,
\[
  -\partial_w\Gamma_\lambda
  =\frac{x_*}{2wA}+\frac{B}{A\Sigma}>0,
\]
and the same argument applies to the $z$ coordinate.

\paragraph{(5) Strict submodularity.}
With the notation above, the first-order conditions imply
\[
  \lambda-\frac1\Sigma
  =\frac{w}{2x_*A}
  =\frac{z}{2y_*B}.
\]
Consequently, define the parameter
\[
  \tau_*:=\frac{x_*A}{w}=\frac{y_*B}{z}>0.
\]
For $u,\tau>0$, set
\begin{equation}\label{eq:futau-def}
    f_u(\tau):=\frac{u+\sqrt{u^2+4u\tau}}2.
\end{equation}
Then $x_*=-w+f_w(\tau_*)$, $y_*=-z+f_z(\tau_*)$, $A=f_w(\tau_*)$, $B=f_z(\tau_*)$, and
\begin{equation}\label{eq:tau-eq}
    g(\tau_*;w,z)=0,
  \qquad
  g(\tau;w,z):=
  \frac1{2\tau}+\frac1{f_w(\tau)+f_z(\tau)}-\lambda.
\end{equation}
The explicit formula for $f_u$ shows that it increases in both arguments.
Hence $\partial_\tau g<0$ and $\partial_zg<0$, and implicit differentiation gives $\partial_z\tau_*=-(\partial_zg)/(\partial_\tau g)<0$.  
Moreover,
\[
  \Sigma=\left(\lambda-\frac1{2\tau_*}\right)^{-1},
  \qquad A=f_w(\tau_*).
\]
Since $d\Sigma/d\tau=-\Sigma^2/(2\tau^2)<0$, we have $\partial_z\Sigma>0$ and $\partial_zA=f_w'(\tau_*)\partial_z\tau_*<0$.  
Differentiating \eqref{eq:partialwGammalambda} yields
\[
  \partial_{wz}^2\Gamma_\lambda
  =-\frac{\partial_z\Sigma}{\Sigma^2}
    +\frac{\partial_zA}{2A^2}<0,
\]
which proves \eqref{eq:submodularity}.
\end{proof}

\subsection{Upper bounding $\sum_{i=1}^{q-1}\Gamma_\lambda(\omega_i,\omega_{i+1})$}
\label{sec:path-estimate-proof}
In this section, we prove \cref{lem:path-estimate}.
\paragraph{Step 1. Symmetrization.}

For a positive vector $v=(v_1,\ldots,v_N)\in\mathbb{R}^N$ and $0\le k\le\lfloor N/2\rfloor$, let $\mathcal M_k(N)$ be the set of all collections of $k$ disjoint unordered pairs from $\{1,\ldots,N\}$.  
Define the largest sum over all $k$-edge matchings on the coordinates of $v$ to be
\[
  \Phi_{k,\lambda}(v_1,\ldots,v_N)
  :=\max_{M\in\mathcal M_k(N)}
  \sum_{\{a,b\}\in M}\Gamma_\lambda(v_a,v_b),
  \qquad \Phi_{0,\lambda}:=0.
\]
Let $N=q-1$ and
\[
  k_q:=\left\lfloor\frac{q-1}{2}\right\rfloor.
\]
We are going to show that
\begin{equation}
\label{eq:parity-bound}
  \sum_{i=1}^{q-1}\Gamma_\lambda(\omega_i,\omega_{i+1})
  \le
  2\Phi_{k_q,\lambda}
  (\omega_q^\downarrow,\ldots,\omega_2^\downarrow)
  +C_{\alpha,\lambda}.
\end{equation}

We first split the consecutive pairs into an odd matching and an even matching, according to whether $i$ is odd or even.
An illustration is given in \cref{fig:odd-even-matchings}.
\begin{figure}[H]
  \centering
  \includegraphics[width=0.82\textwidth]{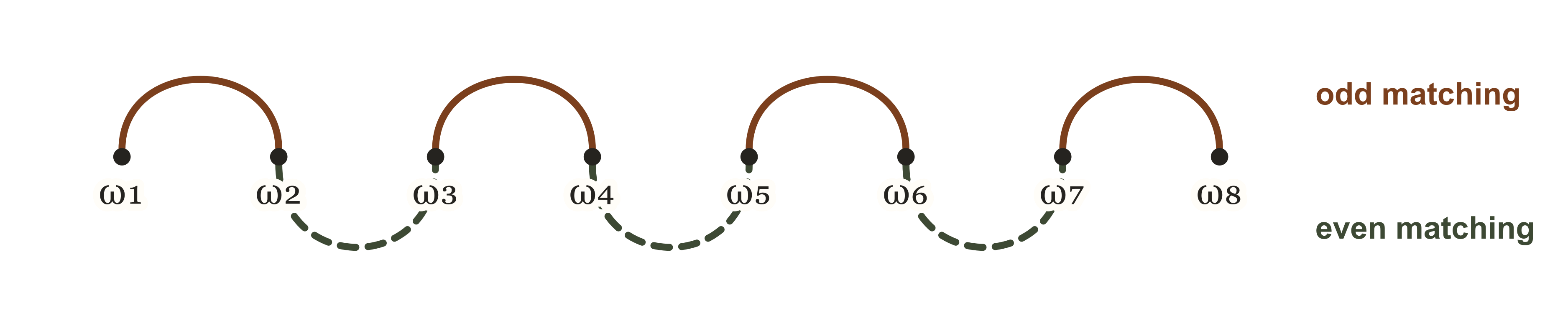}
  \caption{The consecutive pairs split into two matchings with $q=8$.}
  \label{fig:odd-even-matchings}
\end{figure}
Choose $i_*$ such that $\omega_{i_*}=\omega_1^\downarrow$.  
If $q$ is odd, both matchings have $k_q$ pairs and leave one index unused.
If $i_*$ is matched, replace it by the unused index.  
The unused coordinate has no larger value, so the coordinatewise decrease in \cref{lem:gamma-geometry} shows that the matching sum cannot decrease.
If $q$ is even, the same replacement applies to the even matching with $k_q$ pairs.  
The odd matching has $k_q+1$ pairs. Delete the pair containing $i_*$.  
The remaining matching has $k_q$ pairs drawn from $(\omega_q^\downarrow,\ldots,\omega_2^\downarrow)$, and is therefore bounded by $\Phi_{k_q,\lambda}(\omega_q^\downarrow,\ldots,\omega_2^\downarrow)$.
By \eqref{eq:weight-lower-bound}, every $\omega_i\ge\alpha$, and coordinatewise decrease gives $\Gamma_\lambda(\omega_i,\omega_j)\le\Gamma_\lambda(\alpha,\alpha)$.  
In particular, the contribution of the deleted pair is bounded by $\max\{0,\Gamma_\lambda(\alpha,\alpha)\}$.  
This proves \eqref{eq:parity-bound}.

\paragraph{Step 2. Majorization.}
For a fixed matching, the sum of its $\Gamma_\lambda$ terms is convex and coordinatewise decreasing.  
Taking the maximum over all matchings preserves both properties, and permuting the coordinates only permutes the matchings. 
Thus $\Phi_{k,\lambda}$ is symmetric, convex, and coordinatewise decreasing.
For a symmetric, convex, and coordinatewise decreasing function, we have the following majorization principle:
\begin{lemma}
\label{lem:majorization-comparison}
Let $F:(0,\infty)^N\to\mathbb R$ be symmetric, convex, and coordinatewise decreasing. 
Suppose $a_1\le\cdots\le a_N$ and $b_1\le\cdots\le b_N$ satisfy
\[
  \sum_{i=1}^{\ell}a_i\ge\sum_{i=1}^{\ell}b_i,
  \qquad 1\le\ell\le N.
\]
Then $F(a)\le F(b)$.
\end{lemma}
Here, the partial-sum assumption is precisely \(a\prec^{\,w}b\), in the notation of weak supermajorization. Since every symmetric convex function is Schur-convex and \(F\) is coordinatewise decreasing, the conclusion \(F(a)\le F(b)\) follows from the standard characterization of functions preserving weak supermajorization; see \cite[Theorem 3.A.8 and Proposition 3.C.2]{marshall2011inequalities}.



For $2\le s\le q$, define the comparison weights
\begin{equation}
\label{eq:reference-weights}
  \bar w_s^{(q)}
  :=q^{1+\alpha}\bigl((s-1)^{-\alpha}-s^{-\alpha}\bigr),
\end{equation}
which may also be written as
\[
  \bar w_s^{(q)}
  =q^{1+\alpha}\int_{s-1}^{s}\alpha u^{-1-\alpha}\,du.
\]
Since the integrand is decreasing, $\bar w_q^{(q)}\le\cdots\le\bar w_2^{(q)}$. 
Moreover, for $1\le p<q$,
\[
  \sum_{s=p+1}^q\bar w_s^{(q)}
  =q^{1+\alpha}\int_p^q\alpha u^{-1-\alpha}\,du
  =q\left[\left(\frac qp\right)^\alpha-1\right].
\]
Setting $p=q-\ell$ in \eqref{eq:sequence-constraints} gives, for $1\le\ell\le q-1$,
\[
  \sum_{i=1}^{\ell}\omega_{q+1-i}^\downarrow
  \ge\sum_{i=1}^{\ell}\bar w_{q+1-i}^{(q)}.
\]
Applying \cref{lem:majorization-comparison} to $F=\Phi_{k_q,\lambda}$, $a=(\omega_q^\downarrow,\ldots,\omega_2^\downarrow)$, and $b=(\bar w_q^{(q)},\ldots,\bar w_2^{(q)})$ proves
\begin{equation}
  \Phi_{k_q,\lambda}
  (\omega_q^\downarrow,\ldots,\omega_2^\downarrow)
  \le
  \Phi_{k_q,\lambda}
  (\bar w_q^{(q)},\ldots,\bar w_2^{(q)}).
  \label{eq:majorization-comparison}
\end{equation}

\paragraph{Step 3. Identify the maximizing matching.}
Majorization has replaced the unknown weights by the explicit sequence $(\bar w_q^{(q)},\ldots,\bar w_2^{(q)})$. 
It remains to identify the matching that maximizes the resulting sum.

For $0<a\le b\le c\le d$, symmetry and strict submodularity give
\[
\begin{aligned}
  \Gamma_\lambda(a,d)+\Gamma_\lambda(b,c)
  &\ge\Gamma_\lambda(a,c)+\Gamma_\lambda(b,d),\\
  \Gamma_\lambda(a,d)+\Gamma_\lambda(b,c)
  &\ge\Gamma_\lambda(a,b)+\Gamma_\lambda(c,d).
\end{aligned}
\]
Let $0<v_1\le\cdots\le v_N$.  Since $\Gamma_\lambda$ decreases in each coordinate, a maximizing matching may be chosen to use the $2k$ smallest coordinates. 
If $v_1$ is not paired with $v_{2k}$, exchange the partners of their two pairs.
The two inequalities above show that the total cannot decrease. 
Fixing the pair $(v_1,v_{2k})$ and repeating on the remaining coordinates gives
\begin{equation}
\label{eq:extremal-pairing}
  \Phi_{k,\lambda}(v_1,\ldots,v_N)
  =\sum_{j=1}^k\Gamma_\lambda(v_j,v_{2k+1-j}).
\end{equation}
Applied to the comparison weights, this identity gives
\begin{equation}
  \Phi_{k_q,\lambda}(\bar w_q^{(q)},\ldots,\bar w_2^{(q)})
  =\sum_{j=1}^{k_q}
  \Gamma_\lambda
  (\bar w_{q+1-j}^{(q)},\bar w_{q-2k_q+j}^{(q)}).
  \label{eq:reference-pairing}
\end{equation}

\paragraph{Step 4. Comparison with the integral.}
Finally, we estimate $\sum_{j=1}^{k_q}\Gamma_\lambda(\bar w_{q+1-j}^{(q)},\bar w_{q-2k_q+j}^{(q)})$ by comparing it with its corresponding Riemann integral.
By the mean value theorem, for every $2\le s\le q$ there is $\theta_s\in(s-1,s)$ such that
\begin{equation}
\label{eq:mean-value-weights}
  \bar w_s^{(q)}=W_\alpha(\theta_s/q),
  \qquad
  W_\alpha(t):=\alpha t^{-1-\alpha},
  \quad 0<t\le1.
\end{equation}
To control the error introduced by replacing these mean-value points with the uniform mesh, define
\[
  G(t,v):=\Gamma_\lambda(W_\alpha(t),W_\alpha(v)),
  \qquad 0<t\le1,
  \quad 0< v\le1.
\]
The following regularity estimate shows that one argument of $G$ can be extended to $0$.
\begin{lemma}[Uniform Lipschitz bound]
\label{lem:uniform-lipschitz}
For fixed $\alpha,\lambda>0$, the function $G$ has a Lipschitz extension to $[0,1]\times[1/2,1]$.
\end{lemma}
\begin{proof}[Proof of \cref{lem:uniform-lipschitz}]
Let $D:=(0,1]\times[1/2,1]$. For $(t,v)\in D$, write $w:=W_\alpha(t),z:=W_\alpha(v)$, and let $(x_*,y_*)$ be the unique optimizer in the definition of $\Gamma_\lambda(w,z)$. 
Set $A:=w+x_*,B:=z+y_*,\Sigma:=A+B$ as in the proof of \cref{lem:gamma-geometry}.
The first-order conditions give
\[
\tau_*:=\frac{x_*A}{w}=\frac{y_*B}{z},
\qquad
\lambda=\frac{1}{2\tau_*}+\frac{1}{\Sigma}.
\]
We first obtain bounds on the optimizer that are uniform over $D$. Since $v\in[1/2,1]$,
\[
\alpha\le z\le \alpha 2^{1+\alpha}.
\]
Moreover, as $z+y_*\ge\sqrt{z\tau_*}$,
\[
  \lambda
  \le\frac1{2\tau_*}+\frac1{\sqrt{\alpha\tau_*}}.
\]
The right-hand side tends to $0$ as $\tau_*\to\infty$, so $\tau_*\le C_{\alpha,\lambda}$.  
In particular, $x_*,y_*\le\tau_*\le C_{\alpha,\lambda}$.  
The bounds on $z$ and $y_*$ give $B:=z+y_*\le\alpha2^{1+\alpha}+C_{\alpha,\lambda}$, whereas $A:=w+x_*\ge w$ and $\Sigma:=A+B\ge w$.

We next bound the derivatives of $\Gamma_\lambda$. The derivative identity \eqref{eq:partialwGammalambda} yields
\[
-\partial_w\Gamma_\lambda(w,z)
=\frac{x_*}{2wA}+\frac{B}{A\Sigma}.
\]
Since $A\ge w$, $\Sigma\ge A\ge w$, and $x_*,B\le C_{\alpha,\lambda}$, $\left|\partial_w\Gamma_\lambda(w,z)\right|\le C_{\alpha,\lambda}w^{-2}$.
Similarly,
\[
-\partial_z\Gamma_\lambda(w,z)
=\frac{y_*}{2zB}+\frac{A}{B\Sigma}.
\]
Using $z,B\ge\alpha$, $y_*\le C_{\alpha,\lambda}$, and $A\le\Sigma$, we obtain $\left|\partial_z\Gamma_\lambda(w,z)\right|\le \frac{C_{\alpha,\lambda}}{2\alpha^2}+\frac1\alpha\le C_{\alpha,\lambda}$.

By the chain rule,
\[
\partial_tG(t,v)
=\partial_w\Gamma_\lambda(w,z)\,W_\alpha'(t).
\]
Since $|W_\alpha'(t)|=\alpha(1+\alpha)t^{-2-\alpha}$ and $W_\alpha(t)^{-2}=\alpha^{-2}t^{2+2\alpha}$, we have
\[
|\partial_tG(t,v)|
\le C_{\alpha,\lambda}
|W_\alpha'(t)|W_\alpha(t)^{-2}
\le C_{\alpha,\lambda}t^\alpha
\le C_{\alpha,\lambda}.
\]
Likewise,
\[
\partial_vG(t,v)
=\partial_z\Gamma_\lambda(w,z)\,W_\alpha'(v).
\]
Because $v\in[1/2,1]$, the derivative $W_\alpha'(v)$ is uniformly bounded, and hence
\[
|\partial_vG(t,v)|\le C_{\alpha,\lambda}.
\]
It follows that $G$ is uniformly Lipschitz on $D$. 

In particular, for every $v\in[1/2,1]$, the limit
\[
\overline G(0,v):=\lim_{t\downarrow0}G(t,v)
\]
exists. Defining $\overline G=G$ on $D$ and using these limits on $\{0\}\times[1/2,1]$, the preceding Lipschitz estimate passes to the limit. 
Hence $\overline G$ is a Lipschitz extension of $G$ to $[0,1]\times[1/2,1]$.
\end{proof}

Using this extension, define
\[
  I(t):=G(t,1-t),
  \qquad 0\le t\le\frac12.
\]
After increasing the Lipschitz constant if needed, there are finite $L_{\alpha,\lambda}$ and $M_{\alpha,\lambda}$ such that
\[
\begin{aligned}
  |G(t,v)-G(t',v')|
  &\le L_{\alpha,\lambda}(|t-t'|+|v-v'|),\\
  |I(t)-I(t')|
  &\le L_{\alpha,\lambda}|t-t'|,
  \qquad |I(t)|\le M_{\alpha,\lambda}.
\end{aligned}
\]
Assume first that $q\ge3$, and write
\[
  r_q:=q-2k_q\in\{1,2\},
  \qquad
  u_j:=\frac{r_q+j-1}{q},
  \qquad 1\le j\le k_q.
\]
By symmetry and \eqref{eq:reference-pairing},
\[
  \Phi_{k_q,\lambda}(\bar w_q^{(q)},\ldots,\bar w_2^{(q)})
  =\sum_{j=1}^{k_q}
  G\left(\frac{\theta_{r_q+j}}q,
          \frac{\theta_{q+1-j}}q\right).
\]
Moreover,
\[
  \frac{\theta_{r_q+j}}q
  \in\left(u_j,u_j+\frac1q\right),
  \qquad
  \frac{\theta_{q+1-j}}q
  \in\left(1-u_j+\frac{r_q-1}{q},
            1-u_j+\frac{r_q}{q}\right).
\]
The total displacement from $(u_j,1-u_j)$ is less than $(1+r_q)/q\le3/q$.  Since $k_q\le q/2$,
\[
  \left|
    \Phi_{k_q,\lambda}(\bar w_q^{(q)},\ldots,\bar w_2^{(q)})
    -\sum_{j=1}^{k_q}I(u_j)
  \right|
  \le\frac32L_{\alpha,\lambda}.
\]
The intervals $[u_j-1/q,u_j]$ are consecutive, and at most $1/q$ of $[0,1/2]$ is omitted:  
For odd $q$, they fill $[0,1/2-1/(2q)]$; for even $q$, they fill $[1/q,1/2]$. 
On each covered interval, replacing $I(t)$ by its value at the right endpoint costs at most $L_{\alpha,\lambda}/(2q)$.  
Hence
\[
  \left|
    \sum_{j=1}^{k_q}I(u_j)
    -q\int_0^{1/2}I(t)\,dt
  \right|
  \le\frac14L_{\alpha,\lambda}+M_{\alpha,\lambda}.
\]
Combining the two comparisons gives
\[
\begin{aligned}
  \Phi_{k_q,\lambda}(\bar w_q^{(q)},\ldots,\bar w_2^{(q)})
  &=\sum_{j=1}^{k_q}I(u_j)+O_{\alpha,\lambda}(1)\\
  &=q\int_0^{1/2}
    \Gamma_\lambda(W_\alpha(t),W_\alpha(1-t))\,dt
    +O_{\alpha,\lambda}(1).
\end{aligned}
\]
When $q=2$, $k_q=0$, and the same estimate holds after enlarging the constant. 
Combining this estimate with \eqref{eq:parity-bound} and \eqref{eq:majorization-comparison} proves \eqref{eq:path-estimate}.

\section{Anytime Lower Bound}\label{sec:anytime}

In this section, we further extend the improved lower bound to the anytime case, for which an infinite schedule $h=(h_k)_{k\geq 1}$ is fixed in advance, and its prefix $H_n=(h_1,\ldots,h_n)$ is used at horizon $n$.  Compared with the non-anytime case, the schedules used at different horizons must be consistent. This leads to a sharper $\Omega(n^{-1.2408})$ lower bound. As a high-level conclusion of the proof, we combine the term-by-term estimate from the preceding sections with the anytime transfer argument of \cite{TsaiFatkhullinZhangHe2026}. As before, we first prove the result for nonnegative infinite schedules.

\Cref{lem:anytime-sequence-reduction} reduces the anytime problem to a truncated sequence inequality, similar to \cref{lem:fixed-sequence-reduction}.

\begin{lemma}
\label{lem:anytime-sequence-reduction}
Fix $\alpha>0$.  Suppose that there are constants
$\eta_\alpha\in(0,1)$ and $C_\alpha<\infty$, depending only on $\alpha$,
with the following property:  For every integer $q\geq 2$, every integer
$1\leq \ell\leq \eta_\alpha q$, and all positive sequences
$(x_i)_{i=1}^q$ and $(\omega_i)_{i=1}^q$ such that
\begin{equation}
  \sum_{i=1}^q x_i<q,
  \qquad
  \sum_{s=k+1}^q\omega_s^\downarrow
  \geq q\left[\left(\frac qk\right)^\alpha-1\right]
  \quad (\forall \ell\leq k<q),
  \label{eq:anytime-clean-condition}
\end{equation}
we have
\begin{equation}
  E_{\mathrm{end}}
  \prod_{i=1}^{q-1}
  K_{\omega_i,\omega_{i+1}}(x_i,x_{i+1})
  \leq C_\alpha,
  \label{eq:anytime-clean-target}
\end{equation}
where $E_{\mathrm{end}}$ and $K_{\omega_i,\omega_{i+1}}(x_i,x_{i+1})$ are defined as in \cref{lem:fixed-sequence-reduction}.

Then every nonnegative infinite schedule satisfies $\limsup_{n\to\infty}n^{\frac{2(1+\alpha)}{2+\alpha}}R_n(H_n)>0$.
\end{lemma}

When $\ell=1$, \eqref{eq:anytime-clean-condition} coincides with
\eqref{eq:sequence-constraints}.
For $\ell>1$, the constraints indexed by $1\leq k<\ell$ are removed,
so the $\ell$ largest $\omega_i$ are unconstrained. Thus \eqref{eq:anytime-clean-condition} admits a larger class of sequences $(x_i)_{i=1}^q$ and $(\omega_i)_{i=1}^q$ than \eqref{eq:sequence-constraints}.

To proceed, we adapt the estimation technique in the proof of \cref{lem:path-estimate} to the truncated tail condition \eqref{eq:anytime-clean-condition}. The resulting product bound has exponent $qJ(\alpha,\lambda)+O_{\alpha,\lambda}(\ell+1)$, rather than $qJ(\alpha,\lambda)$ as in \eqref{eq:product-by-J}.

\begin{lemma}\label{lem:truncated-product}
Fix $\alpha,\lambda>0$.  There are constants
$A_{\alpha,\lambda}>0$ and $B_{\alpha,\lambda}\geq0$, depending only on
$\alpha$ and $\lambda$, such that, for every $q\geq2$, every
$1\leq\ell<q$, and all positive sequences $(x_i)_{i=1}^q$ and
$(\omega_i)_{i=1}^q$ satisfying \eqref{eq:anytime-clean-condition},
\begin{equation}
  E_{\mathrm{end}}
  \prod_{i=1}^{q-1}
  K_{\omega_i,\omega_{i+1}}(x_i,x_{i+1})
  \leq
  A_{\alpha,\lambda}
  \exp\!\left\{
    qJ(\alpha,\lambda)+B_{\alpha,\lambda}(\ell+1)
  \right\}.
  \label{eq:truncated-bound}
\end{equation}
\end{lemma}

\Cref{app:anytime-proofs} contains the proof of the lemmas above.

\paragraph{Completing the proof of
\texorpdfstring{\Cref{thm:anytime}.}{Theorem 1.2.}}

Suppose that $J(\alpha,\lambda)<0$, and choose $\eta_\alpha\in(0,1)$ so that
$B_{\alpha,\lambda}\eta_\alpha\leq-J(\alpha,\lambda)/2$. For every $q\geq2$, every integer $1\leq\ell\leq\eta_\alpha q$, and all positive sequences $(x_i)_{i=1}^q$ and $(\omega_i)_{i=1}^q$ satisfying \eqref{eq:anytime-clean-condition}, \cref{lem:truncated-product} gives
\begin{equation}
\label{eq:anytime-uniform-product}
  E_{\mathrm{end}}
  \prod_{i=1}^{q-1}K_{\omega_i,\omega_{i+1}}(x_i,x_{i+1})
  \leq
  A_{\alpha,\lambda}
  e^{qJ(\alpha,\lambda)+B_{\alpha,\lambda}(\ell+1)}
  \leq
  A_{\alpha,\lambda}e^{B_{\alpha,\lambda}}
  e^{qJ(\alpha,\lambda)/2}
  \leq C_{\alpha,\lambda}.
\end{equation}
Thus \eqref{eq:anytime-clean-target} is uniform in $q$ and $\ell$, and \cref{lem:anytime-sequence-reduction} yields, for every nonnegative infinite schedule,
\begin{equation}
\label{eq:anytime-limsup-rate}
  \limsup_{n\to\infty}
  n^{\frac{2(1+\alpha)}{2+\alpha}}R_n(H_n)>0.
\end{equation}

It remains only to find a pair $(\alpha,\lambda)$ for which $J(\alpha,\lambda)<0$.  For $\alpha=0.6342$ and $\lambda=0.4506$, numerical integration gives $J(0.6342,0.4506)<-0.00005<0$.  Since $2(1+0.6342)/(2+0.6342)=1.240756\ldots<1.2408$, no nonnegative infinite schedule satisfies $R_n(H_n)=o(n^{-1.2408})$.
\Cref{app:scalar-variational-problem} explains the optimization over $(\alpha,\lambda)$ and the numerical verification.
We have now proved the nonnegative case of \cref{thm:anytime}.
\Cref{app:negative-stepsizes} further extends it to arbitrary stepsize schedules.

\section{Concluding Remarks}
\label{sec:concluding-remarks}

In this paper, we improve lower bounds for GD with predetermined stepsizes in both the non-anytime and anytime settings. Starting from the hard-function construction and the corresponding product-form lower bound in \cref{lem:hard-product}, we perform a finer term-by-term analysis of the consecutive-pair factors in \eqref{eq:sequence-bound}. In the non-anytime setting, this improves the lower bound in~\cite{Tsai2026} from $\Omega(n^{-1.7321})$ to $\Omega(n^{-1.6342})$. Combined with the finite-to-anytime transfer of~\cite[Theorem~4.1]{TsaiFatkhullinZhangHe2026}, the same analysis improves the anytime lower bound from $\Omega(n^{-4/3})$ to $\Omega(n^{-1.2408})$. Both lower bounds continue to hold when negative stepsizes are allowed.

Closing the gaps between these lower bounds and their corresponding upper bounds remains a significant open problem.
Since our term-by-term estimates used to prove the sequence inequality in \eqref{eq:sequence-bound} appear fairly tight, this suggests that our current use of \cref{lem:hard-product} through \cref{lem:fixed-sequence-reduction} may not suffice to close the gap to the $O(n^{-\log_2(1+\sqrt{2})})$ silver-schedule upper bound~\cite{AltschulerParrilo2023,GrimmerShuWang2025Composing}.
 
To elaborate, while \cref{lem:hard-product} gives a lower bound for every selection of long steps, the proof of \cref{lem:fixed-sequence-reduction} uses only the $q$ largest excesses at the $q$ chosen in \eqref{eq:choice-of-q}. This choice guarantees the scale estimate \eqref{eq:Dq-universal-upper-bound} and the tail inequalities in \eqref{eq:sequence-constraints}, and is therefore sufficient for our proof. However, it is not known to maximize the lower bound in \eqref{eq:hard-product}. It remains open whether optimizing over several values of $q$, or combining the corresponding selections, can yield a stronger analogue of \cref{lem:fixed-sequence-reduction}. A more structural question is whether \cref{lem:hard-product} is already sufficient to prove an $\Omega(n^{-\log_2(1+\sqrt{2})})$ lower bound matching the silver-stepsize rate~\cite{AltschulerParrilo2023}.
We suspect that it is necessary to resort to alternative hard functions with a hierarchy of scales, mirroring the recursive structure of silver stepsizes.

For the anytime case, the exponent $2(1+\alpha)/(2+\alpha)$ arises from the finite-to-anytime transfer in \cref{lem:anytime-sequence-reduction}. 
The argument combines the terminal-step estimate \eqref{eq:app-terminal-step-bound} with the tail bounds in \eqref{eq:app-finite-prefix-bounds} at horizons where the final stepsize is the largest seen so far.
It remains open whether coupling several horizons, or replacing the terminal-step estimate by a multi-step bound, can yield a stronger anytime lower bound.

\section*{AI Disclosure}
For the non-anytime result in \cref{sec:technical-overview-fixed,sec:nonanytime-proof}, we used ChatGPT-5.6 Sol over multiple rounds to help develop the improvement. 
ChatGPT's initial proof was disorganized and difficult to follow. The authors subsequently digested it in detail, worked through the calculation, and reorganized and rewrote the presentation.

After further review, the authors identified that the anytime bound could also be improved by combining the refined non-anytime analysis developed in this paper with the framework of Tsai et al.~\cite{TsaiFatkhullinZhangHe2026}, and both bounds continue to hold when negative stepsizes are allowed. 
We then used ChatGPT-5.6 Sol to help develop the proof of this extension, presented in \cref{sec:anytime,app:anytime-proofs,app:negative-stepsizes}.

The authors take full responsibility for the correctness and originality of all content.

\bibliographystyle{alpha}
\bibliography{ref}

\clearpage
\appendix
\phantomsection
\addcontentsline{toc}{section}{Appendix}
\section{Hard-Function Construction}
\label[appendix]{app:hard-functions}

This appendix recalls the geometric idea behind the construction used in
\cref{lem:hard-product}. The formal realization and all projection
inequalities are proved in~\cite[Theorem~4.1]{MaChen2026}; see
also~\cite{Tsai2026} for a complementary geometric exposition. We first
explain why this construction is a natural response to long steps.

\paragraph{Two motivating facts.}

\begin{fact}[$\Omega(n^{-1})$ Huber lower bound]
\label{huberfact}
For every nonnegative stepsize schedule $H$, the Huber construction gives
\[
  R_n(H)\geq \frac{1}{1+2\sum_{k=1}^n h_k}.
\]
In particular, if $h_k\le1$ for every $1\le k\le n$, then
$R_n(H)\ge 1/(2n+1)$.
\end{fact}

\Cref{huberfact} is proved by setting $\delta=(1+2\sum_k h_k)^{-1}$ and considering the one-dimensional Huber function
\[
  \phi_\delta(u)=
  \begin{cases}
    u^2/2, & |u|\le\delta,\\
    \delta|u|-\delta^2/2, & |u|>\delta.
  \end{cases}
\]
This quadratic--linear function was introduced by Huber~\cite{Huber1964}. Its use as a tight worst-case instance for constant-step gradient descent appears in \cite[Theorem~3.2]{DroriTeboulle2014}. See also \cite[Footnote~1]{AltschulerParrilo2023}.

Starting from $u_1=1$, all GD iterates stay in the affine region $u\geq\delta$, where $\phi_\delta'(u)=\delta$.
Thus $u_{n+1}=1-\delta\sum_k h_k$, and direct substitution gives $2\phi_\delta(u_{n+1})=\delta$. 
Hence a schedule that improves on the $n^{-1}$ scale must use steps strictly larger than 1.

\begin{fact}\label{fact:quadratic-chain}
In the standard quadratic-chain proof of the classical first-order oracle lower bound~\cite[Section~2.1.2, Lemma~2.1.5]{Nesterov2004}, let $\mathcal R_k=\operatorname{span}\{e_1,\ldots,e_k\}$. The tridiagonal hard quadratic satisfies
\[
  x\in\mathcal R_k
  \quad\Longrightarrow\quad
  \nabla f(x)\in\mathcal R_{k+1}.
\]
Consequently, starting from $x_0=0$, induction shows that a first-order method satisfying the usual linear-span condition can reveal at most one new coordinate per oracle call. A chain whose length grows with the horizon yields the classical $\Omega(n^{-2})$ lower bound \cite[Theorem~2.1.7]{Nesterov2004}.
\end{fact}

The relevant idea is the supply of fresh orthogonal directions: after $k$ calls, part of the hard instance still lies outside the subspace reached by the algorithm.

\paragraph{Forcing long steps to switch gradient directions.}
It is useful to view the construction in \cite[Theorem~4.1]{MaChen2026} as a combination of the above two ideas. 
The Huber function supplies a constant-gradient affine region on which the intervening updates make only one-dimensional progress.
Introducing fresh coordinates prevents the selected long steps from repeatedly exploiting the same one-dimensional overshoot. This restricts the convergence rate, thereby establishing a tighter lower bound.

More precisely, fix selected indices
\[
  0<t_1<\cdots<t_q\le n,
  \qquad h_{t_i}>1,
\]
set $t_0=0$ and divide the schedule into $q$ transition blocks and one terminal block. For $1\le i\le q$, transition block $i$ contains $(h_{t_{i-1}+1},\dots,h_{t_i})$, while the terminal block contains $(h_{t_q+1},\dots,h_n)$.
The construction first specifies a desirable trajectory.
Place anchors
\[
  X_i=\lambda_i e_i,
  \qquad 1\le i\le q+1,
\]
where $e_1,\ldots,e_{q+1}$ are orthonormal.
Then, for each $1\le i\le q$, arrange for the iterates to start at the anchor $X_i$, make only one-dimensional progress, and reach the next anchor $X_{i+1}$ at the final selected long step.
As a consequence, the gradient $g_i$ for this block is fixed to be 
\[
g_i=\frac{X_{i}-X_{i+1}}{S_i+h_{t_i}},
\qquad 1\le i\le q.
\]
Note that after the final selected long step lands on the next coordinate-axis anchor $X_{i+1}$, the active gradient changes from $g_i$ to $g_{i+1}$. 
After the last selected step, the remaining iterates use
\[
  g_{q+1}=\frac{\lambda_{q+1}}{1+2S_{q+1}}e_{q+1}.
\]
\Cref{fig:hard-trajectory} illustrates this geometry.

\begin{figure}[t]
  \centering
  \begin{tikzpicture}[
  font=\small,
  intervening/.style={-{Stealth[length=4pt]},draw=shortteal,line width=.9pt},
  long/.style={-{Stealth[length=6pt,width=5pt]},draw=longpurple,line width=2pt},
  iterate/.style={circle,fill=shortteal,inner sep=1.25pt},
  hardanchor/.style={circle,fill=longpurple,draw=white,line width=.5pt,inner sep=2.3pt},
  shortlabel/.style={font=\scriptsize,text=shortteal!70!black,fill=white,
    fill opacity=.92,text opacity=1,inner sep=1.5pt,rounded corners=1pt}
]
  \coordinate (O) at (0,0);
  \coordinate (X1) at (4.35,0);
  \coordinate (X2) at (-2.35,-2.35);
  \coordinate (X3) at (0,3.75);

  \coordinate (P11) at ($(X1)!.055!(X2)$);
  \coordinate (P12) at ($(X1)!.120!(X2)$);
  \coordinate (P13) at ($(X1)!.195!(X2)$);
  \coordinate (P14) at ($(X1)!.275!(X2)$);
  \coordinate (P15) at ($(X1)!.360!(X2)$);
  \coordinate (P21) at ($(X2)!.050!(X3)$);
  \coordinate (P22) at ($(X2)!.115!(X3)$);
  \coordinate (P23) at ($(X2)!.190!(X3)$);
  \coordinate (P24) at ($(X2)!.275!(X3)$);
  \coordinate (P25) at ($(X2)!.365!(X3)$);
  \coordinate (P31) at (0,3.38);
  \coordinate (P32) at (0,2.94);
  \coordinate (P33) at (0,2.53);
  \coordinate (P34) at (0,2.04);
  \coordinate (F) at (0,1.50);

  \draw[-{Stealth[length=5pt]},softgray] (O)--(4.95,0)
    node[right] {$e_1$};
  \draw[-{Stealth[length=5pt]},softgray] (O)--(-3.05,-3.05)
    node[below left] {$e_2$};
  \draw[-{Stealth[length=5pt]},softgray] (O)--(0,4.40)
    node[above] {$e_3$};
  \node[below right=2pt and 3pt] at (O) {$0$};

  \draw[intervening] (X1)--(P11);
  \draw[intervening] (P11)--(P12);
  \draw[intervening] (P12)--(P13);
  \draw[intervening] (P13)--(P14);
  \draw[intervening] (P14)--(P15);
  \path (X1)--(P15)
    node[midway,below=7pt,sloped,shortlabel] {$h_1,\ldots,h_{t_1-1}$};
  \draw[long] (P15)--(X2)
    node[midway,below,sloped,text=longpurple!70!black] {$h_{t_1}>1$};

  \draw[intervening] (X2)--(P21);
  \draw[intervening] (P21)--(P22);
  \draw[intervening] (P22)--(P23);
  \draw[intervening] (P23)--(P24);
  \draw[intervening] (P24)--(P25);
  \path (X2)--(P25)
    node[midway,above=7pt,sloped,shortlabel]
      {$h_{t_1+1},\ldots,h_{t_2-1}$};
  \draw[long] (P25)--(X3)
    node[midway,left=5pt,text=longpurple!70!black] {$h_{t_2}>1$};

  \draw[intervening] (X3)--(P31);
  \draw[intervening] (P31)--(P32);
  \draw[intervening] (P32)--(P33);
  \draw[intervening] (P33)--(P34);
  \draw[intervening] (P34)--(F);
  \path (X3)--(F)
    node[midway,right=8pt,shortlabel] {$h_{t_2+1},\ldots,h_n$};

  \foreach \p in {P11,P12,P13,P14,P15,P21,P22,P23,P24,P25,
                   P31,P32,P33,P34,F}
    \node[iterate] at (\p) {};
  \foreach \p in {X1,X2,X3}
    \node[hardanchor] at (\p) {};

  \node[above=4pt] at (X1) {$X_1=\lambda_1e_1$};
  \node[below right=3pt and 2pt] at (X2) {$X_2=\lambda_2e_2$};
  \node[above right=2pt and 5pt] at (X3) {$X_3=\lambda_3e_3$};
  \node[right=5pt] at (F) {$x_{n+1}$};
\end{tikzpicture}
  \caption{Only the selected long steps (purple arrows) switch the gradient directions.}
  \label{fig:hard-trajectory}
\end{figure}
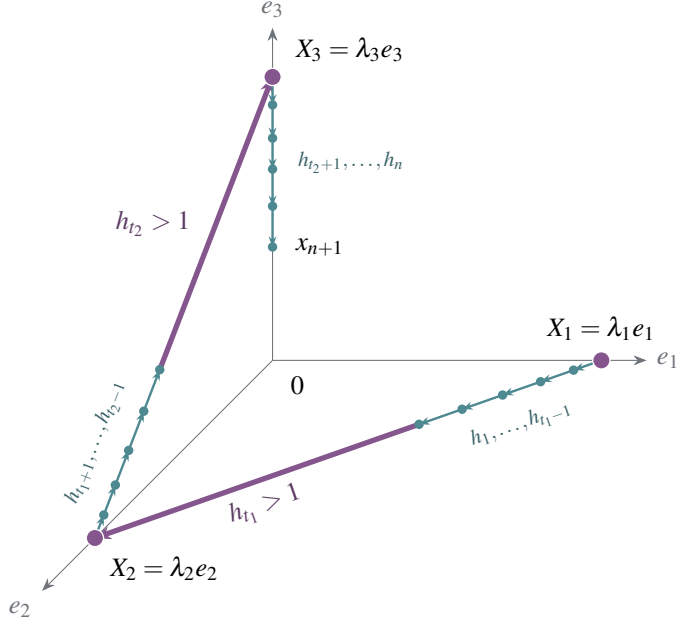

\paragraph{Construction of the hard function via a Moreau envelope.}
The remaining question is how to realize this prescribed piecewise-constant
gradient pattern with one globally defined smooth convex function. 
Intuitively, this is possible by gluing Huber functions together. 
The selected long steps provide enough room for the landscape of the hard function to change while maintaining its smoothness.

More precisely, let
\[
  C:=\operatorname{conv}\{0,g_1,\ldots,g_{q+1}\},
\]
and consider its support function 
\[
\sigma_C(z):=\max_{g\in C}\langle g,z\rangle.
\]
The support function stores the candidate block gradients as slopes.
Now take the Moreau envelope of the support function
$\sigma_C$,
\[
  f_C(x):=
  \min_z\left\{\sigma_C(z)+\frac12\lVert x-z\rVert^2\right\}.
\]
Since $\sigma_C$ is proper, closed, and convex, its Moreau envelope $f_C$ is convex and has a $1$-Lipschitz gradient~\cite[Propositions~5.b, 7.b, and~7.d]{Moreau1965}.
The hard function $f_C$ satisfies the distance and projection identities~\cite[Lemma~4.2]{MaChen2026}
\[
  f_C(x)=\frac12\lVert x\rVert^2-\frac12d(x,C)^2,
  \qquad
  \nabla f_C(x)=\Pi_C(x).
\]
Completing the square also gives
\[
  f_C(x)=
  \max_{g\in C}\left\{\langle g,x\rangle-\frac12\lVert g\rVert^2\right\}.
\]
This hard function can be regarded as a multidimensional analogue of the Huber function. 
On $C$ it equals $\frac12\lVert x\rVert^2$. 
On a normal-cone region where $\Pi_C(x)=g_i$, it equals $\langle g_i,x\rangle-\frac12\lVert g_i\rVert^2$ and is therefore affine.
The trajectory stays in these constant-projection regions between selected long steps. 
In this sense, the construction retains the quadratic region $C$ and affine outer pieces of the one-dimensional Huber function while using new coordinate axes to keep the effects of several selected long
steps separate.

Hence prescribing $g_i$ as the gradient throughout a block reduces to the geometric condition
\[
  g_i=\Pi_C(x)
  \quad\Longleftrightarrow\quad
  \langle x-g_i,v-g_i\rangle\le0
  \quad\text{for every }v\in C.
\]
Since $C$ is the convex hull of finitely many generators, it is enough to check this inequality at $0,g_1,\ldots,g_{q+1}$. 
The proof in~\cite[Theorem~4.1]{MaChen2026} parameterizes the positive scales by the amplitude ratios $\gamma_i:=\lambda_{i+1}/\lambda_i$. 
That theorem gives an explicit admissible upper bound for each $\gamma_i^2$. Choosing the ratios within these ranges ensures the projection inequalities along every intended block.
Saturating the explicit upper bounds on $\gamma_i$ maximizes the final gap $f_C(x_{n+1})-f_C(0)$ and yields the lower bound in \cref{lem:hard-product}. 

\section{Properties of \texorpdfstring{$J(\alpha,\lambda)$}{J(alpha, lambda)} and Numerical Solution}
\label[appendix]{app:scalar-variational-problem}

The proofs of \cref{thm:fixed-horizon,thm:anytime} reduce to finding the smallest possible $\alpha>0$ for which $J(\alpha,\lambda)\le0$ for some $\lambda>0$; see \eqref{eq:product-by-J}--\eqref{eq:fixed-horizon-rate}.\footnote{For the anytime case, the inequality must be strict; see \eqref{eq:anytime-uniform-product}--\eqref{eq:anytime-limsup-rate}.}
Accordingly, for each exponent parameter $\alpha$, we optimize over the scalar parameter $\lambda$ in \eqref{eq:J-definition} and define
\begin{equation}
\label{eq:optimized-J}
  \hat J(\alpha):=\inf_{\lambda>0}J(\alpha,\lambda).
\end{equation}
It remains to find the smallest $\alpha>0$ for which $\hat J(\alpha)\le0$.
The following proposition gives the properties needed to determine $\alpha$ and $\lambda$.

\begin{proposition}
\label{prop:J-variational-properties}
For every $\alpha>0$, the function $\lambda\mapsto J(\alpha,\lambda)$ is strictly convex and has a unique minimizer $\lambda_\star(\alpha)\in(1/4,1/2)$.
The function $\hat J$ is continuous and strictly decreasing.
Hence it has a unique zero $\alpha_\star$ on every interval where it changes sign.
\end{proposition}

\begin{proof}
Let $u_\lambda=(x_\lambda,y_\lambda)$ be the optimizer defining $\Gamma_\lambda(w,z)$, and let $H$ be the Hessian of $\log K_{w,z}$ at $u_\lambda$.
By \cref{lem:gamma-geometry}, this optimizer is unique and $H$ is negative definite.
Applying the implicit function theorem to the first-order condition $\nabla_{x,y}\log K_{w,z}(u_\lambda)=\lambda\mathbf 1$ shows that $u_\lambda$ is differentiable in $\lambda$ and that $u_\lambda'=H^{-1}\mathbf 1$.
Differentiating the optimized value then gives
\[
\begin{aligned}
  \partial_\lambda\Gamma_\lambda(w,z)
  &=-\mathbf 1^\top u_\lambda=-(x_\lambda+y_\lambda),\\
  \partial_{\lambda\lambda}\Gamma_\lambda(w,z)
  &=-\mathbf 1^\top H^{-1}\mathbf 1>0.
\end{aligned}
\]
Consequently, $J(\alpha,\cdot)$ is strictly convex.

To determine the signs of $\partial_\lambda J$ at the endpoints
$\lambda=1/4$ and $\lambda=1/2$, we need bounds on the optimizer $(x,y)$.
Set $a:=x/(w+x)$ and $b:=y/(z+y)$.  The first-order conditions give
\[
  \lambda(x+y)
  =1-\frac{(a-b)^2}{2(a+b-2ab)}.
\]
Since
$a+b-2ab-(a-b)^2=a(1-a)+b(1-b)>0$, this identity yields
$1/2<\lambda(x+y)\le1$.
Equality in the upper bound forces $a=b$.
Writing $\tau=x(w+x)/w=y(z+y)/z$, we have
$w=\tau(1-a)^2/a$ and $z=\tau(1-b)^2/b$.
Thus equality forces $w=z$.  Therefore
\begin{equation}
\label{eq:optimizer-sum-bound}
  \frac12<\lambda(x+y)\le1.
\end{equation}
Moreover, the upper bound is strict whenever $w\ne z$.

By \eqref{eq:optimizer-sum-bound}, $x+y\le1/\lambda$, so $|\partial_\lambda\Gamma_\lambda|$ has a uniform bound on every compact $\lambda$-interval.
We may therefore differentiate under the integral in \eqref{eq:J-definition} to obtain
\begin{equation}
\label{eq:J-lambda-derivative}
  \partial_\lambda J(\alpha,\lambda)
  =2-2\int_0^{1/2}
  \bigl(x_{\alpha,\lambda}(t)+y_{\alpha,\lambda}(t)\bigr)\,dt,
\end{equation}
where $(x_{\alpha,\lambda}(t),y_{\alpha,\lambda}(t))$ is the optimizer corresponding to $(w,z)=(W_\alpha(t),W_\alpha(1-t))$.
It follows from \eqref{eq:J-lambda-derivative} that $\partial_\lambda J(\alpha,1/4)<0$ and $\partial_\lambda J(\alpha,1/2)>0$.
The latter inequality is strict because equality in $\lambda(x+y)\le1$ can occur only when $w=z$, which here happens only at $t=1/2$.
Thus the unique minimizer lies in $(1/4,1/2)$.

For fixed $\lambda$, $W_\alpha(t)$ increases strictly with $\alpha$, while $\Gamma_\lambda$ decreases strictly in each coordinate.
Therefore $J(\alpha,\lambda)$ decreases strictly with $\alpha$.
If $0<\alpha_1<\alpha_2$, evaluating $J(\alpha_2,\cdot)$ at $\lambda_\star(\alpha_1)$ gives $\hat J(\alpha_2)<\hat J(\alpha_1)$.

It remains to prove continuity.  The estimates in the proof of
\cref{lem:uniform-lipschitz} are locally uniform in $(\alpha,\lambda)$.
Consequently, the integrand in \eqref{eq:J-definition} extends continuously
to $t=0$, locally uniformly on compact parameter sets, and is locally bounded
there.  Together with the implicit-function argument above, dominated
convergence shows that $J$ is jointly continuous.
Since the minimizer always lies in the fixed compact interval $[1/4,1/2]$, minimizing over that interval gives continuity of $\hat J$.
\end{proof}

\paragraph{Numerical solution.}
For $u,\tau>0$, write
\[
  X_u(\tau):=\frac{2\tau}{1+\sqrt{1+4\tau/u}}
  =f_u(\tau)-u.
\]
For fixed $\alpha,\lambda>0$ and $0<t\le1/2$, put
$w=W_\alpha(t)$ and $z=W_\alpha(1-t)$.
The optimizer defining $\Gamma_\lambda(w,z)$ is
$x=X_w(\tau)$ and $y=X_z(\tau)$, where $\tau$ is the unique root of
\eqref{eq:tau-eq}.  At $t=0$, we use the continuous extension established
above.
At this root,
\begin{equation}
\label{eq:Gamma-stable-evaluation}
  \Gamma_\lambda(w,z)
  =\log\left[
    \tau\left(\frac1{f_w(\tau)}+\frac1{f_z(\tau)}\right)
  \right]
  -\lambda\bigl(X_w(\tau)+X_z(\tau)\bigr).
\end{equation}
To evaluate $J(\alpha,\lambda)$ and
$\partial_\lambda J(\alpha,\lambda)$, we solve \eqref{eq:tau-eq} at each
quadrature node and use \eqref{eq:Gamma-stable-evaluation} and
\eqref{eq:J-lambda-derivative}.  For fixed $\alpha$,
$\partial_\lambda J(\alpha,\cdot)$ is continuous and strictly increasing
and has opposite signs at $1/4$ and $1/2$.  Bisection therefore gives
$\lambda_\star(\alpha)$. A second bisection applied to the continuous,
strictly decreasing function $\hat J$ gives its zero $\alpha_\star$.

Floating-point computation gives
\begin{equation}
\label{eq:floating-critical-pair}
  \alpha_\star\approx0.63416058454,
  \qquad
  \lambda_\star(\alpha_\star)\approx0.45060731.
\end{equation}
These digits describe the optimized pair but are not used as a rigorous certificate.
The proof uses the nearby pair $(\alpha,\lambda)=(0.6342,0.4506)$.
A computer-assisted calculation using Arb with directed outward rounding
certifies
\begin{equation}
\label{eq:certified-J-value}
  J(0.6342,0.4506)<-0.0000512188<-0.00005<0,
\end{equation}
which is the inequality used in the proofs of
\cref{thm:fixed-horizon,thm:anytime}.

\section{Proofs for the Anytime Lower Bound}
\label[appendix]{app:anytime-proofs}

This appendix proves \cref{lem:anytime-sequence-reduction,lem:truncated-product}.

\subsection{Proof and Intuition of \cref{lem:anytime-sequence-reduction}}
\label{app:proof-anytime-sequence-reduction}

We first explain how the proof of \cref{lem:anytime-sequence-reduction} modifies the finite-to-anytime transfer of \cite[Theorem~4.1]{TsaiFatkhullinZhangHe2026}.
We begin with two lower bounds from one-dimensional hard instances that drive this transfer.
For a prefix \(H_n=(h_1,\ldots,h_n)\), the first one relates \(R_n(H_n)\) to the total stepsize \(h_{1:n}\).
Using convex quadratics, \cite[Lemma~2.1]{TsaiFatkhullinZhangHe2026} gives
\begin{equation}
  R_n(H_n)\ge\frac{1}{4(1+2h_{1:n})}.
  \label{eq:app-quadratic-total-bound}
\end{equation}
We instead use the standard Huber loss from \cref{huberfact}, which gives the stronger bound \(R_n(H_n)\ge \frac{1}{1+2h_{1:n}}\), as stated in \eqref{eq:Huber}.
In particular, if \(R_n(H_n)=o(n^{-p})\) for some \(p>1\), then \(h_{1:n}/n\to\infty\), and the stepsizes are unbounded.

The second bound comes from \cite[Lemma~3.1]{TsaiFatkhullinZhangHe2026}.
Its hard instance is the asymmetric Huber loss
\begin{equation}
  \psi_{\varepsilon,\delta}(u):=
  \begin{cases}
    \delta u-\delta^2/2, & u\ge\delta,\\
    u^2/2, & -\varepsilon\le u\le\delta,\\
    -\varepsilon u-\varepsilon^2/2, & u\le-\varepsilon.
  \end{cases}
  \label{eq:app-asymmetric-huber}
\end{equation}
It is quadratic near the origin and affine on either side.
By choosing \(\varepsilon\) and \(\delta\) so that the update with \(h_m>1\) crosses the quadratic region, their lemma proves
\begin{equation}
  R_n(H_n)\ge
  \frac{(h_m-1)^2}
  {(1+h_{1:m-1})^2(1+2h_{m+1:n})}.
  \label{eq:app-asymmetric-huber-bound}
\end{equation}
We only use the case \(m=n\), which gives
\begin{equation}
  h_n-1\le(1+h_{1:n-1})\sqrt{R_n(H_n)}
  \qquad\text{when }h_n>1.
  \label{eq:app-terminal-step-bound}
\end{equation}

The proof of the \(\Omega(n^{-4/3})\) anytime lower bound in \cite[Theorem~4.1]{TsaiFatkhullinZhangHe2026} combines \eqref{eq:app-quadratic-total-bound}, \eqref{eq:app-asymmetric-huber-bound}, and their Lemmas~2.2 and~2.3.
The latter use convex quadratics to control the total stepsize through the largest stepsize.
\Cref{fig:tsai-anytime-roadmap} summarizes their argument.

\begin{figure}[H]
  \centering
  \begin{tikzpicture}[
    >=Latex,
    ingredient/.style={draw=selectedpurple, fill=selectedpurple!9,
      rounded corners=2pt, align=center, minimum height=1.34cm,
      inner sep=3.5pt},
    sideingredient/.style={ingredient, text width=4.05cm},
    middleingredient/.style={ingredient, text width=6.35cm},
    source/.style={draw=softgray, fill=white, rounded corners=2pt,
      align=center, minimum height=0.68cm, inner xsep=7pt,
      font=\small},
    conclusion/.style={draw=tailgold, fill=tailgold!18, very thick,
      rounded corners=2pt, align=center, minimum height=0.95cm,
      text width=5.65cm},
    connector/.style={thick, draw=softgray},
    flow/.style={-{Latex[length=2.2mm]}, thick, draw=softgray}
  ]
    \node[sideingredient] (total) at (-5.6,0)
      {Total-sum bound~\eqref{eq:app-quadratic-total-bound}\\[2pt]
       \scriptsize cf.~\cite[Lemma~2.1]{TsaiFatkhullinZhangHe2026}};
    \node[middleingredient] (largest) at (0,0)
      {Largest-step control:\\[1pt]
       Bounding \(h_{1:n}\) in terms of
       \(\max_{k\le n}h_k\)\\[2pt]
       \scriptsize cf.~\cite[Lemmas~2.2--2.3]{TsaiFatkhullinZhangHe2026}};
    \node[sideingredient] (terminal) at (5.6,0)
      {Terminal-step bound~\eqref{eq:app-terminal-step-bound}\\[2pt]
       \scriptsize cf.~\cite[Lemma~3.1]{TsaiFatkhullinZhangHe2026}};

    \coordinate (qright) at ($(largest.north)+(0,0.55)$);
    \coordinate (qleft) at (total.north |- qright);
    \coordinate (sourcelevel) at ($(qleft)!0.5!(qright)+(0,0.78)$);
    \node[source] (quadratic) at (sourcelevel)
      {Convex quadratics};
    \node[source] (huber) at (terminal.north |- sourcelevel)
      {Asymmetric Huber loss~\eqref{eq:app-asymmetric-huber}};

    \draw[connector] (quadratic.south) -- (quadratic.south |- qleft);
    \draw[connector] (qleft) -- (qright);
    \draw[flow] (qleft) -- (total.north);
    \draw[flow] (qright) -- (largest.north);
    \draw[flow] (huber.south) -- (terminal.north);

    \coordinate (join) at (0,-1.15);
    \node[conclusion] (barrier) at (0,-2.05)
      {\(\Omega(n^{-4/3})\) anytime lower bound for \(R_n\)\\[1pt]
       \footnotesize cf.~\cite[Theorem~4.1]{TsaiFatkhullinZhangHe2026}};

    \draw[connector] (total.south) |- (join);
    \draw[connector] (largest.south) -- (join);
    \draw[connector] (terminal.south) |- (join);
    \fill[softgray] (join) circle (1.1pt);
    \draw[flow] (join) -- (barrier.north);
  \end{tikzpicture}
  \caption{Roadmap of the proof of
  \cite[Theorem~4.1]{TsaiFatkhullinZhangHe2026}; see also
  \cite[Figure~2]{TsaiFatkhullinZhangHe2026}.}
  \label{fig:tsai-anytime-roadmap}
\end{figure}
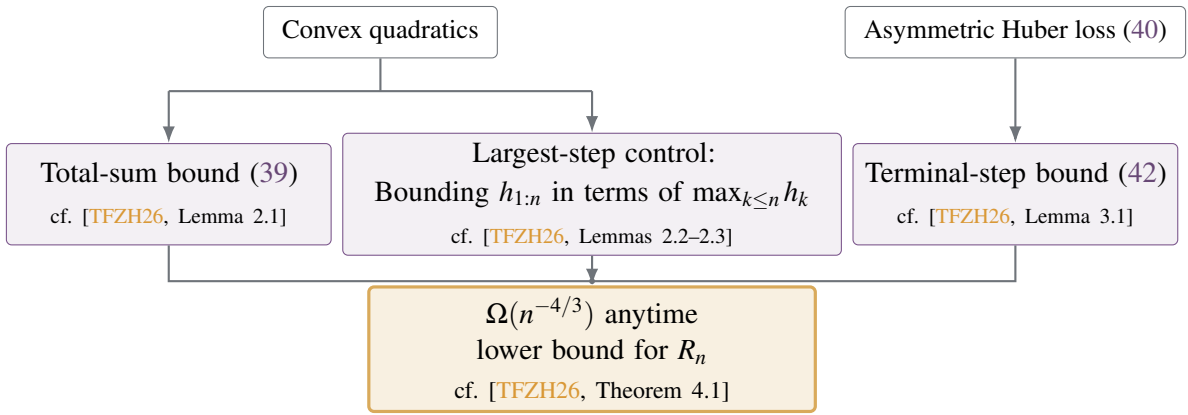

Our proof uses the stronger total-sum bound \eqref{eq:Huber} in place of \cite[Lemma~2.1]{TsaiFatkhullinZhangHe2026}, retains the terminal-step bound \eqref{eq:app-terminal-step-bound} from \cite[Lemma~3.1]{TsaiFatkhullinZhangHe2026}, and replaces the use of \cite[Lemmas~2.2 and~2.3]{TsaiFatkhullinZhangHe2026} with the two more refined tail bounds in \eqref{eq:app-finite-prefix-bounds}, valid for every \(m\) such that \(R_n(H)D_m\le c_0\).
The proof of these tail bounds uses the same hard function and variable substitution as the proof of \cref{lem:fixed-sequence-reduction}.

\begin{proof}[Proof of \cref{lem:anytime-sequence-reduction}]
Fix \(\alpha>0\), and assume that \eqref{eq:anytime-clean-condition} implies \eqref{eq:anytime-clean-target}, with the same constants \(\eta_\alpha,C_\alpha\) for every integer \(q\ge 2\) and every integer \(1\le \ell\le \eta_\alpha q\).

Fix a nonnegative finite schedule \(H=(h_1,\ldots,h_n)\).
Recall the notation from \cref{sec:fixed-sequence-reduction-proof}.
Let \(a_1\ge\cdots\ge a_r>0\) be the decreasing rearrangement of the positive excesses \((h_k-1)_+\), and define
\[
  B:=1+\sum_{k=1}^n\min\{h_k,1\},
  \qquad
  D_s:=B+\sum_{j=s+1}^r a_j,
  \qquad 1\le s\le r.
\]

\medskip
\noindent\textbf{Part I: Constructing tail bounds.}\par
We first prove that there are constants \(c_0,C>0\), depending only on \(\alpha\), such that any \(1\le m\le r\) satisfying \(R_n(H)D_m\le c_0\) also satisfies
\begin{equation}
  m a_m\le C D_m,
  \qquad
  D_m\le C(n+1)^{1+\alpha}m^{-\alpha}.
  \label{eq:app-finite-prefix-bounds}
\end{equation}
Together, these two bounds play the role of \cite[Lemmas~2.2 and~2.3]{TsaiFatkhullinZhangHe2026} in our proof.
The condition \(R_n(H)D_m\le c_0\) implies \(D_m\le c_0/R_n(H)\).
The first inequality compares \(D_{m-1}\) and \(D_m\), since \(D_{m-1}=D_m+a_m\le(1+C/m)D_m\), while the second controls \(D_m\) itself.
Applying both bounds with \(m=j\) gives \(a_j\le C(n+1)^{1+\alpha}j^{-1-\alpha}\), whose sum controls the remaining
smaller excesses.

We now prove the two inequalities in \eqref{eq:app-finite-prefix-bounds}.
Note that only the proof of the second inequality uses the hypothesis that \eqref{eq:anytime-clean-condition} implies \eqref{eq:anytime-clean-target}.
If \(r=0\), there is nothing to prove.
We may therefore assume \(r\ge1\). For the moment, let \(c_0>0\),
and let \(m\) be arbitrary such that \(R_n(H)D_m\le c_0\). We collect the required restrictions on \(c_0\) at the end of Part~I.
The \(q=1\) case in the proof of \cref{lem:fixed-sequence-reduction} gives \(R_n(H)D_1\ge1/4\).
Under the restriction \(c_0<1/4\), the case \(m=1\) cannot satisfy
\(R_n(H)D_m\le c_0\). We henceforth consider only \(2\le m\le r\).

Fix \(2\le q\le r\).
Following the notation in the proof of \cref{lem:fixed-sequence-reduction}, let \(t_1<\cdots<t_q\) be the locations of the \(q\) largest positive excesses, with \(t_0=0\) and \(t_{q+1}=n+1\). For $1\le i\le q+1$, set $S_i:=h_{t_{i-1}+1:t_i-1}$; for $1\le i\le q$, define
\[
  x_i:=\frac{q(S_i+1)}{D_q},
  \qquad
  \omega_i:=\frac{q(h_{t_i}-1)}{D_q},
  \qquad 1\le i\le q,
\]
with $\sum_{i=1}^q x_i<q$.
Let \(G\) be the right-hand side of \eqref{eq:hard-product} for these
selected indices. Then \cref{lem:hard-product} gives \(R_n(H)\ge G\), and
\eqref{eq:top-q-lower-bound} yields
\begin{equation}
  R_n(H)\ge
  \frac{q}{D_q E_{\mathrm{end}}
  \prod_{i=1}^{q-1}K_{\omega_i,\omega_{i+1}}(x_i,x_{i+1})},
  \label{eq:app-anytime-hard-bound}
\end{equation}

\smallskip
\noindent\emph{Step 1. Prove the first inequality in \eqref{eq:app-finite-prefix-bounds}.}
Set \(\beta_\alpha:=1+\max\{e,\alpha\}\). We claim that
\begin{equation}
  s a_s\le \beta_\alpha D_s,
  \qquad m\le s\le r.
  \label{eq:app-marginal-excess-bound}
\end{equation}

For every \(m\le s\le r\), we have \(D_s\le D_m\), and hence \(R_n(H)D_s\le c_0\).
Apply \eqref{eq:app-anytime-hard-bound} with \(q=s\) and use the \(s\) largest excesses to form \(x_1,\ldots,x_s\) and \(\omega_1,\ldots,\omega_s\).
Every selected excess is at least \(a_s\), so \(\omega_i\ge s a_s/D_s\) for every \(1\le i\le s\).

If \(s a_s>\beta_\alpha D_s\), then \(\omega_i>\beta_\alpha\) for every \(i\).
Define \(\Gamma_1\)\footnote{Here
\(\Gamma_1=\left.\Gamma_\lambda\right|_{\lambda=1}\).} by
\[
  \Gamma_1(w,z):=
  \sup_{x,y>0}\{\log K_{w,z}(x,y)-(x+y)\}.
\]
For the \(x_i,\omega_i\) just obtained with \(q=s\), put \(\Delta_s:=s-\sum_{i=1}^s x_i>0\).
Since \(\min_i\omega_i>\beta_\alpha\), the endpoint estimate in the proof of \cref{lem:envelope-reduction} gives
\begin{equation}
  E_{\mathrm{end}}e^{-(2\Delta_s+x_1+x_s)}
  \le A_{\beta_\alpha},
  \label{eq:app-rankwise-endpoint}
\end{equation}
for some \(A_{\beta_\alpha}<\infty\) depending only on \(\beta_\alpha\).
Combining \eqref{eq:app-rankwise-endpoint} with the definition of
\(\Gamma_1\) gives
\[
  E_{\mathrm{end}}
  \prod_{i=1}^{s-1}K_{\omega_i,\omega_{i+1}}(x_i,x_{i+1})
  \le
  A_{\beta_\alpha}\exp\!\left\{
    2s+\sum_{i=1}^{s-1}\Gamma_1(\omega_i,\omega_{i+1})
  \right\}.
\]
Because \(\Gamma_1\) decreases in each coordinate and \(\Gamma_1(\beta_\alpha,\beta_\alpha)=-\log\beta_\alpha-1\),
\[
  2s+\sum_{i=1}^{s-1}\Gamma_1(\omega_i,\omega_{i+1})
  \le 2s+(s-1)(-\log\beta_\alpha-1).
\]
Put \(\delta:=\log\beta_\alpha-1>0\). Then $2s+(s-1)(-\log\beta_\alpha-1)=-\delta s+\log\beta_\alpha+1$.
The preceding product is therefore at most \(\widetilde A_{\beta_\alpha}e^{-\delta s}\), where \(\widetilde A_{\beta_\alpha}:=e\beta_\alpha A_{\beta_\alpha}\).
Applying \eqref{eq:app-anytime-hard-bound} with \(q=s\) gives
\[
  R_n(H)
  \ge\frac{s}{D_s E_{\mathrm{end}}
  \prod_{i=1}^{s-1}K_{\omega_i,\omega_{i+1}}(x_i,x_{i+1})}
  \ge\frac{s e^{\delta s}}{\widetilde A_{\beta_\alpha}D_s}
  \ge\frac{s e^{\delta s}}{\widetilde A_{\beta_\alpha}c_0}R_n(H),
\]
If \(\widetilde A_{\beta_\alpha}c_0<1\), then \(s\ge m\ge2\) gives \(s e^{\delta s}/(\widetilde A_{\beta_\alpha}c_0)>1\), a contradiction. Hence \eqref{eq:app-marginal-excess-bound}, and thus the first inequality in \eqref{eq:app-finite-prefix-bounds}, holds under this condition.

\smallskip
\noindent\emph{Step 2. Prove the second inequality in \eqref{eq:app-finite-prefix-bounds}.}

Assume \(\widetilde A_{\beta_\alpha}c_0<1\), so that Step~1 applies.
Put \(F_s:=D_s s^\alpha\), and choose
\[
  \widehat q\in\operatorname*{arg\,min}_{m\le s\le r}F_s.
\]
Using \eqref{eq:app-marginal-excess-bound}, \(D_{s-1}=D_s+a_s\), and \(\log(1+u)\le u\), we obtain, for every \(m<q\le r\),
\[
  \log\frac{D_m}{D_q}
  =\sum_{s=m+1}^q\log\left(1+\frac{a_s}{D_s}\right)
  \le \beta_\alpha\sum_{s=m+1}^q\frac1s
  \le \beta_\alpha\log\frac qm.
\]
Therefore
\[
  \frac{D_q}{D_m}\ge\left(\frac mq\right)^{\beta_\alpha}.
\]
We now prove that \(F_m\le A F_r\) for a constant \(A\) depending only on \(\alpha\).
Take \(A\ge\eta_\alpha^{-(\beta_\alpha-\alpha)}\).
If \(F_m>A F_r\), then \(\widehat q>m\) and
\[
  \frac{F_{\widehat q}}{F_m}<\frac1A,
  \qquad
  \frac{F_{\widehat q}}{F_m}
  =\frac{D_{\widehat q}}{D_m}
  \left(\frac{\widehat q}{m}\right)^\alpha
  \ge\left(\frac m{\widehat q}\right)^{\beta_\alpha-\alpha}.
\]
Hence \(m/\widehat q<A^{-1/(\beta_\alpha-\alpha)}\le\eta_\alpha\), so
\(m<\eta_\alpha\widehat q\).
Because the selected excesses are \(a_1,\ldots,a_{\widehat q}\), the decreasing rearrangement of the normalized weights satisfies \(\omega_j^\downarrow=\widehat q a_j/D_{\widehat q}\) for
\(1\le j\le\widehat q\).
For every \(m\le k<\widehat q\), the minimizing property of \(\widehat q\) gives
\begin{equation}
  \sum_{s=k+1}^{\widehat q}\omega_s^\downarrow
  =\frac{\widehat q}{D_{\widehat q}}
  \sum_{s=k+1}^{\widehat q}a_s
  =\widehat q\left(\frac{D_k}{D_{\widehat q}}-1\right)
  \ge \widehat q
  \left[\left(\frac{\widehat q}{k}\right)^\alpha-1\right].
  \label{eq:app-truncated-tail-from-minimizer}
\end{equation}
This inequality, \(m\le\eta_\alpha\widehat q\), and the bound \(\sum_{i=1}^{\widehat q}x_i<\widehat q\) together verify \eqref{eq:anytime-clean-condition} with \(q=\widehat q\) and \(\ell=m\).
Applying \eqref{eq:anytime-clean-target} gives
\[
  E_{\mathrm{end}}
  \prod_{i=1}^{\widehat q-1}
  K_{\omega_i,\omega_{i+1}}(x_i,x_{i+1})
  \le C_\alpha.
\]
Combining this estimate with \eqref{eq:app-anytime-hard-bound} and \(D_{\widehat q}\le D_m\le c_0/R_n(H)\) gives
\[
  R_n(H)\ge\frac{\widehat q}{C_\alpha D_{\widehat q}}
  \ge\frac{\widehat q}{C_\alpha c_0}R_n(H),
\]
If, in addition, \(C_\alpha c_0<1\), then \(\widehat q/(C_\alpha c_0)>1\), a contradiction. 
Hence \(F_m\le A F_r\) under these conditions.
Since \(D_r=B\le n+1\) and \(r\le n\),
\[
  D_m m^\alpha=F_m
  \le A F_r
  =A D_r r^\alpha
  \le A(n+1)n^\alpha.
\]
Collecting the conditions used above, choose, once and for all,
\[
  0<c_0<
  \min\left\{
    \frac14,
    \widetilde A_{\beta_\alpha}^{-1},
    C_\alpha^{-1}
  \right\}.
\]
For this choice, the preceding argument proves both inequalities in \eqref{eq:app-finite-prefix-bounds} for every \(m\) satisfying \(R_n(H)D_m\le c_0\). 
Since \(\beta_\alpha,\widetilde A_{\beta_\alpha},C_\alpha\), and \(A\) depend only on \(\alpha\), so does \(c_0\). 
Taking \(C:=\max\{\beta_\alpha,A\}\) completes Part~I.

\medskip
\noindent\textbf{Part II: The anytime transfer.}\par
We now combine Part I with \eqref{eq:Huber} and \eqref{eq:app-terminal-step-bound}.
The remaining proof adapts the finite-to-anytime transfer in \cite[Proof of Theorem~4.1]{TsaiFatkhullinZhangHe2026}.
Fix a nonnegative infinite schedule.
For each prefix \(H_n\), define \(a_1,\ldots,a_r,B,D_s\) as in Part I.

Suppose for contradiction that
\[
  R_n(H_n)=o\!\left(n^{-2(1+\alpha)/(2+\alpha)}\right).
\]
Since \(\alpha>0\), the exponent \(2(1+\alpha)/(2+\alpha)\) is greater than \(1\). 
The contradiction hypothesis and \eqref{eq:Huber} therefore imply \(h_{1:n}/n\to\infty\). 
Hence the stepsizes are unbounded, and there are arbitrarily large indices \(n\) satisfying
\begin{equation}
  h_n=\max_{1\le k\le n}h_k>1.
  \label{eq:app-prefix-maximum}
\end{equation}
Moreover, since \(D_r=B\le n+1\),
\[
  R_n(H_n)D_r
  \le (n+1)R_n(H_n)
  =o\!\left(n^{-\alpha/(2+\alpha)}\right)
  \longrightarrow 0.
\]
We may therefore fix one such \(n\), sufficiently large that both \eqref{eq:app-prefix-maximum} and \(R_n(H_n)D_r\le c_0\) hold. 
Since \(h_n\) is the largest stepsize in \(H_n\), its excess is the largest, and hence \(a_1=h_n-1\).

Define the cutoff\footnote{We explain the choice of \(m_\star\) in \cref{remark:mstar}.}
\begin{equation}
  m_\star:=\left\lfloor
  \frac{1}{4\sqrt{R_n(H_n)}}
  \right\rfloor.
  \label{eq:app-m-star}
\end{equation}
Since \(R_n(H_n)\to0\) and \(\lfloor x\rfloor\ge x/2\) for \(x\ge2\), our sufficiently large choice of \(n\) satisfies
\begin{equation}
  m_\star\ge\frac{1}{8\sqrt{R_n(H_n)}}.
  \label{eq:app-m-star-lower}
\end{equation}
If the number of long steps \(r\le m_\star\), the terminal-step bound \eqref{eq:app-terminal-step-bound} gives
\[
  \sum_{j=1}^r a_j
  \le m_\star a_1
  \le\frac14(1+h_{1:n}).
\]
Since \(h_{1:n}=(B-1)+\sum_{j=1}^r a_j\) and \(B-1\le n\), this would imply \(h_{1:n}=O(n)\), a contradiction.
Thus \(r>m_\star\).

To apply Part~I with \(m=m_\star\), we first prove that \(R_n(H_n)D_{m_\star}\le c_0\).
Suppose instead that \(R_n(H_n)D_{m_\star}>c_0\).
Since \(r>m_\star\), \(R_n(H_n)D_r\le c_0\), and \(D_s\) decreases with \(s\), there is a smallest index \(q_\star>m_\star\) such that
\[
  R_n(H_n)D_{q_\star}\le c_0.
\]
Applying both inequalities in \eqref{eq:app-finite-prefix-bounds} with \(m=q_\star\) gives
\[
  D_{q_\star-1}=D_{q_\star}+a_{q_\star}
  \le\left(1+\frac C{q_\star}\right)D_{q_\star}
  \le C(n+1)^{1+\alpha}q_\star^{-\alpha}.
\]
Using \(q_\star>m_\star\) and \eqref{eq:app-m-star-lower}, we obtain
\begin{align*}
  R_n(H_n)D_{q_\star-1}
  &\le C
  n^{1+\alpha}R_n(H_n)^{1+\alpha/2}\\
  &=C\left(
  n^{\frac{2(1+\alpha)}{2+\alpha}}R_n(H_n)
  \right)^{(2+\alpha)/2}
  =o(1),
\end{align*}
contradicting the minimality of \(q_\star\), which gives \(R_n(H_n)D_{q_\star-1}>c_0\).
Hence
\begin{equation}
  R_n(H_n)D_{m_\star}\le c_0.
  \label{eq:app-mstar-small-product}
\end{equation}

For every \(m_\star<j\le r\), \(D_j\le D_{m_\star}\), so \(R_n(H_n)D_j\le c_0\).
Applying \eqref{eq:app-finite-prefix-bounds} with \(m=j\) gives
\[
  a_j\le\frac{C D_j}{j}
  \le C(n+1)^{1+\alpha}j^{-1-\alpha}.
\]
For \(1\le j\le m_\star\), we have \(a_j\le a_1=h_n-1\), so the terminal-step bound \eqref{eq:app-terminal-step-bound} controls the first sum below.
For \(m_\star<j\le r\), sum the preceding estimate and use \eqref{eq:app-m-star-lower} for the second sum:
\begin{align*}
  \sum_{j=1}^{m_\star}a_j
  &\le m_\star a_1
  \le\frac14(1+h_{1:n}),\\
  \sum_{j=m_\star+1}^{r}a_j
  &\le C(n+1)^{1+\alpha}
  \sum_{j=m_\star+1}^{\infty}j^{-1-\alpha}\\
  &\le C(n+1)^{1+\alpha}m_\star^{-\alpha}
  \le C(n+1)^{1+\alpha}R_n(H_n)^{\alpha/2}.
\end{align*}
Using \(h_{1:n}=(B-1)+\sum_j a_j\) and \(B-1\le n\), we absorb \(\tfrac14 h_{1:n}\) into the left-hand side to obtain
\begin{equation}
  h_{1:n}\le
  C\left(n+n^{1+\alpha}R_n(H_n)^{\alpha/2}\right).
  \label{eq:app-total-stepsize-bound}
\end{equation}
Multiplying by \(R_n(H_n)\) gives
\begin{align*}
  R_n(H_n)h_{1:n}
  &\le C\left(
    nR_n(H_n)+n^{1+\alpha}R_n(H_n)^{1+\alpha/2}
  \right)\\
  &=C\left(
    nR_n(H_n)+
    \left(n^{\frac{2(1+\alpha)}{2+\alpha}}R_n(H_n)\right)^{(2+\alpha)/2}
  \right)
  =o(1).
\end{align*}
But \eqref{eq:Huber} gives \(R_n(H_n)(1+2h_{1:n})\ge1\), a contradiction.
Therefore
\[
  \limsup_{n\to\infty}
  n^{\frac{2(1+\alpha)}{2+\alpha}}R_n(H_n)>0.
\]
\end{proof}

\begin{remark}\label{remark:mstar}
We explain the choice of \(m_\star\) in \eqref{eq:app-m-star}.
For the chosen \(n\), \(h_n=\max_{1\le k\le n}h_k>1\), so \(a_1=h_n-1\). 
The terminal-step bound \eqref{eq:app-terminal-step-bound} gives $a_1\le(1+h_{1:n})\sqrt{R_n(H_n)}$.
Since \(a_j\le a_1\), every cutoff \(m\) satisfies
\[
  \sum_{j=1}^{m}a_j
  \le m(1+h_{1:n})\sqrt{R_n(H_n)}.
\]
The head estimate only requires \(m_\star\le \varepsilon R_n(H_n)^{-1/2}\), where \(\varepsilon\in(0,1)\) is fixed. 
Indeed, the preceding inequality then gives \(\sum_{j=1}^{m_\star}a_j\le\varepsilon(1+h_{1:n})\), which can be absorbed in \eqref{eq:app-total-stepsize-bound} since \(\varepsilon<1\).
In \eqref{eq:app-m-star}, we take \(\varepsilon=1/4\). Part~I then gives
\[
  \sum_{j>m_\star}a_j
  \le C(n+1)^{1+\alpha}m_\star^{-\alpha}
  \le C(n+1)^{1+\alpha}R_n(H_n)^{\alpha/2}.
\]
Combining the head and tail estimates with \(h_{1:n}=(B-1)+\sum_{j\le m_\star}a_j+\sum_{j>m_\star}a_j\) recovers \eqref{eq:app-total-stepsize-bound}.
\end{remark}

\subsection{Proof of \cref{lem:truncated-product}}
\label{app:proof-truncated-product}

We adapt the proof of \cref{lem:path-estimate} to establish \eqref{eq:truncated-bound}.

\begin{proof}
\medskip
\noindent\textbf{Step 1. Eliminate the \(x_i\)'s.}\par

Taking \(k=q-1\) in \eqref{eq:anytime-clean-condition} gives \(\min_i\omega_i\ge\alpha\), the only consequence of the tail conditions used in the proof of \cref{lem:envelope-reduction}. 
Hence the same proof gives
\[
  E_{\mathrm{end}}
  \prod_{i=1}^{q-1}K_{\omega_i,\omega_{i+1}}(x_i,x_{i+1})
  \le C_{\alpha,\lambda}
  \exp\!\left\{
    2\lambda q+
    \sum_{i=1}^{q-1}\Gamma_\lambda(\omega_i,\omega_{i+1})
  \right\}.
\]
It remains to bound the sum in the exponent.

\medskip

\noindent\textbf{Step 2. Set aside the \(\ell\) unconstrained largest weights.}\par
We first split the consecutive-pair sum from Step~1 into an odd matching and an even matching, according to whether \(i\) is odd or even.
\[
  (\omega_1,\omega_2),(\omega_3,\omega_4),\ldots
  \qquad\text{and}\qquad
  (\omega_2,\omega_3),(\omega_4,\omega_5),\ldots.
\]
The tail inequalities in \eqref{eq:anytime-clean-condition} do not constrain the \(\ell\) largest weights \(\omega_1^\downarrow,\ldots,\omega_\ell^\downarrow\).
Mark their positions in the chronological sequence and delete, from each matching, every pair containing one of these positions.
Each marked position belongs to at most one pair in each matching, so at most \(2\ell\) terms are deleted in total.

Let \(M_0\) and \(M_1\) denote the remaining odd and even matchings, respectively, and put
\[
  k_\nu:=|M_\nu|,
  \qquad
  d_\nu:=q-2k_\nu.
\]
Thus \(d_\nu\) is the number of positions not used by \(M_\nu\).
Since \(M_\nu\) contains none of the \(\ell\) marked positions, \(d_\nu\ge\ell\).
Before any pair is deleted, each of the odd and even matchings leaves at most two positions unused.
Deleting a pair creates two additional unused positions, and at most \(\ell\) pairs are deleted from either matching.
Hence
\begin{equation}
  \ell\le d_\nu\le2\ell+2.
  \label{eq:app-unused-positions}
\end{equation}
Recall the comparison weights \(\bar w_s^{(q)}\) from
\eqref{eq:reference-weights}. They satisfy
\(\bar w_q^{(q)}\le\cdots\le\bar w_2^{(q)}\) and, for every
\(\ell\le m<q\),
\[
  \sum_{s=m+1}^q\bar w_s^{(q)}
  =q\left[\left(\frac qm\right)^\alpha-1\right].
\]
Hence \eqref{eq:anytime-clean-condition} gives, for every \(1\le j\le q-\ell\),
\begin{equation}
  \sum_{s=q-j+1}^q\omega_s^\downarrow
  \ge
  \sum_{s=q-j+1}^q\bar w_s^{(q)}.
  \label{eq:app-truncated-majorization-tails}
\end{equation}

\medskip

\noindent\textbf{Step 3. Identify the maximizing matching and compare it with the integral.}\par
After ordering the remaining weights increasingly, each \(M_\nu\) is a \(k_\nu\)-edge matching on \((\omega_q^\downarrow,\ldots,\omega_{\ell+1}^\downarrow)\).
The proof of \cref{lem:path-estimate} shows that \(\Phi_{k_\nu,\lambda}\) is symmetric, convex, and coordinatewise decreasing.
Hence \cref{lem:majorization-comparison}, \eqref{eq:app-truncated-majorization-tails}, and \eqref{eq:extremal-pairing} give, for \(\nu\in\{0,1\}\),
\begin{equation}
\begin{aligned}
  \sum_{\{i,j\}\in M_\nu}
  \Gamma_\lambda(\omega_i,\omega_j)
  &\le
  \Phi_{k_\nu,\lambda}
  (\omega_q^\downarrow,\ldots,\omega_{\ell+1}^\downarrow)\\
  &\le
  \Phi_{k_\nu,\lambda}
  (\bar w_q^{(q)},\ldots,\bar w_{\ell+1}^{(q)})\\
  &=
  \sum_{j=1}^{k_\nu}\Gamma_\lambda\!\left(
    \bar w_{q+1-j}^{(q)},
    \bar w_{d_\nu+j}^{(q)}
  \right).
\end{aligned}
  \label{eq:app-truncated-extremal-pairing}
\end{equation}

We next compare the explicit matching sum in \eqref{eq:app-truncated-extremal-pairing} with its corresponding Riemann integral.
Recall that \(W_\alpha(t)=\alpha t^{-1-\alpha}\) for \(0<t\le1\).
By \cref{lem:uniform-lipschitz},
\[
  g_{\alpha,\lambda}(t)
  :=\Gamma_\lambda(W_\alpha(t),W_\alpha(1-t)),
  \qquad 0<t\le\frac12,
\]
extends continuously to \(t=0\). Define
\[
  I_{\alpha,\lambda}
  :=\int_0^{1/2}g_{\alpha,\lambda}(t)\,dt.
\]
We use the following shifted Riemann-sum estimate.
For every integer \(0\le k\le\lfloor(q-1)/2\rfloor\), set \(d:=q-2k\), so \(d\ge1\).
Then
\begin{equation}
  \sum_{j=1}^{k}
  \Gamma_\lambda\!\left(
    \bar w_{q+1-j}^{(q)},
    \bar w_{d+j}^{(q)}
  \right)
  \le qI_{\alpha,\lambda}+C_{\alpha,\lambda}(d+1).
  \label{eq:app-shifted-riemann}
\end{equation}
Set \(M_{\alpha,\lambda}:=\max\{0,\Gamma_\lambda(\alpha,\alpha)\}\).
To prove \eqref{eq:app-shifted-riemann}, first suppose \(d\ge q/4\).
Since every \(\bar w_s^{(q)}\ge\alpha\), each summand is at most \(M_{\alpha,\lambda}\).
Therefore
\[
  \sum_{j=1}^{k}
  \Gamma_\lambda(\bar w_{q+1-j}^{(q)},\bar w_{d+j}^{(q)})
  -qI_{\alpha,\lambda}
  \le q\left(\frac{M_{\alpha,\lambda}}2+|I_{\alpha,\lambda}|\right)
  \le C_{\alpha,\lambda}d.
\]
Now suppose \(d<q/4\).
By the mean-value theorem, for each \(2\le s\le q\), there is \(\xi_s\in(s-1,s)\) such that
\[
  \bar w_s^{(q)}=W_\alpha(\xi_s/q).
\]
The Lipschitz estimate in \cref{lem:uniform-lipschitz} applies to \((u,v)\mapsto\Gamma_\lambda(W_\alpha(u),W_\alpha(v))\) on \([0,1/2]\times[1/2,1]\).

For each \(j\) such that \(u_j:=(d+j)/q\le1/2\),
\[
  \left|\frac{\xi_{d+j}}q-u_j\right|\le\frac1q,
  \qquad
  \left|\frac{\xi_{q+1-j}}q-(1-u_j)\right|
  \le\frac{d+1}{q}.
\]
Symmetry and the Lipschitz estimate give
\[
  \Gamma_\lambda\!\left(
    \bar w_{q+1-j}^{(q)},\bar w_{d+j}^{(q)}
  \right)
  \le
  g_{\alpha,\lambda}(u_j)
  +C_{\alpha,\lambda}\frac{d+1}{q}.
\]
Since \(1\le j\le k=(q-d)/2\), there are at most \(d+1\) indices with \(u_j>1/2\).
For these indices, both \(\xi_{d+j}/q\) and \(\xi_{q+1-j}/q\) lie in \([1/4,1]\), because \(d<q/4\).
Thus both arguments of \(\Gamma_\lambda\) lie in \([\alpha,\alpha\,4^{1+\alpha}]\), and the total contribution of these terms is at most \(C_{\alpha,\lambda}(d+1)\).
For the other indices, the grid \(u_j\) has mesh \(1/q\) and begins at \((d+1)/q\).
The portion of \([0,1/2]\) not covered by the corresponding Riemann cells has total length at most \((d+1)/q\).
The boundedness and Lipschitz continuity of \(g_{\alpha,\lambda}\) therefore give
\[
  \sum_{j:\,u_j\le1/2}g_{\alpha,\lambda}(u_j)
  \le qI_{\alpha,\lambda}+C_{\alpha,\lambda}(d+1).
\]
Summing the error \(C_{\alpha,\lambda}(d+1)/q\) over these indices and adding the at most \(d+1\) terms with \(u_j>1/2\) proves \eqref{eq:app-shifted-riemann}.

Apply \eqref{eq:app-shifted-riemann} to the odd and even matchings \(M_0\) and \(M_1\), with \((k,d)=(k_\nu,d_\nu)\).
Add back the at most \(2\ell\) terms deleted in Step~2.
Since \(\omega_i\ge\alpha\) for every \(i\) by Step~1, each deleted term is at most \(M_{\alpha,\lambda}\).
Since \(d_\nu\le2\ell+2\), we obtain
\begin{equation}
  \sum_{i=1}^{q-1}\Gamma_\lambda(\omega_i,\omega_{i+1})
  \le 2qI_{\alpha,\lambda}
  +C_{\alpha,\lambda}(\ell+1).
  \label{eq:app-truncated-gamma-path}
\end{equation}
Combining the estimate from Step~1 with \eqref{eq:app-truncated-gamma-path} and using \(J(\alpha,\lambda)=2\lambda+2I_{\alpha,\lambda}\) from \eqref{eq:J-definition} proves \eqref{eq:truncated-bound}.
\end{proof}

\section{Extension to Possibly Negative Stepsize Schedules}
\label[appendix]{app:negative-stepsizes}

The preceding analysis proves the lower bounds for nonnegative stepsize schedules.
We now remove this restriction.
Given a possibly negative stepsize schedule, we associate with it a nonnegative schedule obtained from the running maxima of its partial sums.
The lemma below shows that the lower bounds on $R_n$ used in the preceding proofs remain valid for $R_n(\widetilde H)$ when their right-hand sides are computed from this associated nonnegative schedule.

\begin{lemma}
\label{lem:running-maximum-lifting}
Let $\widetilde H=(\widetilde h_1,\ldots,\widetilde h_n)\in\mathbb R^n$.
Set $P_0=M_0=0$ and define
\begin{equation}
  P_k:=\sum_{j=1}^k\widetilde h_j,
  \qquad
  M_k:=\max_{0\le j\le k}P_j,
  \qquad
  h_k:=M_k-M_{k-1}.
  \label{eq:running-maximum-schedule}
\end{equation}
Equivalently, $h_k=(P_k-M_{k-1})_+$ and $\sum_{j=1}^k h_j=M_k$. Thus $H$ advances only when the signed cumulative stepsize $P_k$ reaches a new maximum.
Then $H=(h_1,\ldots,h_n)$ is nonnegative, and the following holds.
\begin{enumerate}[label=\textup{(\roman*)}]
  \item For every $0<t_1<\cdots<t_q\le n$ with $h_{t_i}>1$, the right-hand side of \eqref{eq:hard-product}, with all quantities computed from $H$, is a lower bound on $R_n(\widetilde H)$.
  \item We have
  \begin{equation}
    R_n(\widetilde H)
    \ge \frac{1}{1+2M_n}
    =\frac{1}{1+2\sum_{k=1}^n h_k}.
    \label{eq:signed-huber}
  \end{equation}
  \item If $h_n>1$, then
  \begin{equation}
    R_n(\widetilde H)
    \ge\left(\frac{h_n-1}{1+M_{n-1}}\right)^2
    =\left(
      \frac{h_n-1}{1+\sum_{k=1}^{n-1}h_k}
    \right)^2.
    \label{eq:signed-terminal-step}
  \end{equation}
\end{enumerate}
\end{lemma}

\pagebreak[3]
\begin{proof}
\medskip\noindent\textbf{Part (i).}\enspace
Fix a selection of long steps of $H$, and let $f_C$ be the hard function used to prove \cref{lem:hard-product}.
See \cref{app:hard-functions} for its construction. The corresponding $H$-trajectory\footnote{By the $H$-trajectory, we mean the GD trajectory on $f_C$ generated from the prescribed initial point by the schedule $H$. The $\widetilde H$-trajectory is defined analogously.} is illustrated in \cref{fig:hard-trajectory}.
To transfer the lower bound, we show that the $\widetilde H$-trajectory follows the same sequence of gradient blocks as the $H$-trajectory on $f_C$. The two trajectories coincide at each transition from the $i$th block to the $(i+1)$st.

We use the notation of \cref{app:hard-functions}. Recall that $C=\operatorname{conv}\{0,g_1,\ldots,g_{q+1}\}$ and $f_C(x)=\max_{g\in C}\{\langle g,x\rangle-\frac12\lVert g\rVert^2\}$. For every $x$, the maximum of $\langle g,x\rangle-\frac12\lVert g\rVert^2$ over $g\in C$ is attained uniquely at $g=\Pi_C(x)$, and we have $\nabla f_C(x)=\Pi_C(x)$. Moreover, $\Pi_C(y)=g_i$ if and only if $\langle y-g_i,v-g_i\rangle\le0$ for every $v\in C$.
We first prove that shifting a point in the direction $+g_i$ preserves its projection whenever that projection is $g_i$.
The formulas for the block gradients in \cref{app:hard-functions} give $\langle g_i,g_j\rangle\le0$ whenever $i\ne j$.
It follows that $\langle g_i,v-g_i\rangle\le0$ for every vertex of $C$, and hence for every $v\in C$.
Thus, if $\Pi_C(y)=g_i$ and $a\ge0$, then
\[
  \langle y+ag_i-g_i,v-g_i\rangle
  =\langle y-g_i,v-g_i\rangle
   +a\langle g_i,v-g_i\rangle
  \le0
  \qquad\text{for every }v\in C.
\]
The displayed inequality is precisely the projection condition for $g_i$. Therefore,
\begin{equation}
  \Pi_C(y+ag_i)=g_i.
  \label{eq:projection-preservation}
\end{equation}
Since $\Pi_C(y)=\Pi_C(y+ag_i)=g_i$ and the maximum of $\langle g,x\rangle-\frac12\lVert g\rVert^2$ over $g\in C$ is attained uniquely at $g=\Pi_C(x)$, we have
\begin{equation}
  f_C(y+ag_i)=\langle g_i,y+ag_i\rangle-\frac12\lVert g_i\rVert^2=f_C(y)+a\lVert g_i\rVert^2.
  \label{eq:function-shift}
\end{equation}

We now use \eqref{eq:projection-preservation} to compare the trajectories generated by $H$ and $\widetilde H$ on the same hard function $f_C$.
Set $d_k:=M_k-P_k\ge0$.
The quantity $d_k$ measures how far $P_k$ lies below its running maximum, and the definitions give
\[
  d_k=d_{k-1}+h_k-\widetilde h_k,
  \qquad
  h_k>0\Longrightarrow d_k=0.
\]
Let $(\bar x_k)$ and $(\widetilde x_k)$ be the trajectories generated by $H$ and $\widetilde H$, respectively, on the hard function $f_C$ and from the same initial point.
We prove by induction over the blocks and the updates within each block that, whenever update $k$ lies in the $i$th block,
\begin{equation}
  \widetilde x_k=\bar x_k+d_{k-1}g_i.
  \label{eq:trajectory-offset}
\end{equation}
For the first update, \eqref{eq:trajectory-offset} is immediate.
Suppose it holds at update $k$ in the $i$th block. Since the $H$-trajectory has gradient $g_i$, \eqref{eq:projection-preservation} shows that the $\widetilde H$-trajectory also has gradient $g_i$. Hence
\[
  \widetilde x_{k+1}
  =\bar x_{k+1}
   +(d_{k-1}+h_k-\widetilde h_k)g_i
  =\bar x_{k+1}+d_kg_i,
\]
where the last equality uses the identity for $d_k$ above.
At a selected index $t_i$, the condition $h_{t_i}>1$ means that the running maximum increases. Hence $P_{t_i}=M_{t_i}$ and $d_{t_i}=0$. The same identity also gives
\[
  \widetilde h_{t_i}=d_{t_i-1}+h_{t_i},
  \qquad
  \widetilde x_{t_i+1}=\bar x_{t_i+1}.
\]
The two trajectories therefore meet at the boundary between the $i$th and $(i+1)$st blocks, so the induction continues with gradient $g_{i+1}$.

Applying the induction through update $n$ gives $\widetilde x_{n+1}=\bar x_{n+1}+d_ng_{q+1}$. By the construction of $f_C$, $\Pi_C(\bar x_{n+1})=g_{q+1}$, and therefore \eqref{eq:function-shift} gives
\[
  f_C(\widetilde x_{n+1})
  =f_C(\bar x_{n+1})+d_n\lVert g_{q+1}\rVert^2
  \ge f_C(\bar x_{n+1}).
\]
Both trajectories are generated on the same hard function from the same initial point, with the same minimizer.
Together with the preceding inequality, this shows that, on $f_{\mathcal C}$, the normalized gap generated by $\widetilde H$ is at least that generated by $H$. The latter is bounded below by the right-hand side of \eqref{eq:hard-product}. Since $R_n(\widetilde H)$ is the supremum over all admissible functions and initial points, the same lower bound holds for $R_n(\widetilde H)$. This proves Part~\textup{(i)}.

\vspace{1cm}

The remaining two statements follow from one-dimensional Huber functions.

\medskip\noindent\textbf{Part (ii).}\enspace
For \eqref{eq:signed-huber}, take $\delta=(1+2M_n)^{-1}$ in the Huber function from \cref{huberfact} and start from $u_1=1$.
An induction shows that, for every $1\le k\le n$,
\[
  u_k=1-\delta P_{k-1}\ge1-\delta M_n\ge\delta,
\]
and hence every queried gradient is $\delta$.
Thus $u_{n+1}=1-\delta P_n\ge1-\delta M_n$.
With minimizer $u^\star=0$, we have $\frac12|u_1-u^\star|^2=1/2$. Since $1-\delta M_n\ge\delta$ and $\phi_\delta$ is increasing on $[\delta,\infty)$, the definition of $R_n$ gives
\[
  R_n(\widetilde H)
  \ge2\phi_\delta(u_{n+1})
  \ge2\phi_\delta(1-\delta M_n)
  =\delta.
\]

\medskip\noindent\textbf{Part (iii).}\enspace
Suppose $h_n>1$.
Then $P_n=M_n=M_{n-1}+h_n$.
Set
\[
  \delta:=\frac{1}{1+M_{n-1}},
  \qquad
  \varepsilon:=\delta(h_n-1),
\]
and use the asymmetric Huber function in \eqref{eq:app-asymmetric-huber}.
Since $P_{k-1}\le M_{n-1}$ for every $1\le k\le n$, we have $u_k=1-\delta P_{k-1}\ge\delta$, so every queried gradient equals $\delta$, while $u_{n+1}=1-\delta P_n=-\varepsilon$.
Consequently,
\[
  R_n(\widetilde H)
  \ge2\psi_{\varepsilon,\delta}(-\varepsilon)
  =\varepsilon^2,
\]
which proves \eqref{eq:signed-terminal-step}.
\end{proof}

\paragraph{Completing the non-anytime proof for possibly negative stepsizes.}
Fix $\widetilde H\in\mathbb R^n$ and define $H$ by \eqref{eq:running-maximum-schedule}.
In the proof of \cref{lem:fixed-sequence-reduction}, Parts (i) and (ii) of \cref{lem:running-maximum-lifting} give the required lower bounds on $R_n(\widetilde H)$.
Their right-hand sides are exactly the same as those of the corresponding bounds on $R_n(H)$, with every quantity computed from $H$.
After these replacements, every subsequent step concerns only $H$ and is unchanged.
This completes the proof of \cref{thm:fixed-horizon} for every $\widetilde H\in\mathbb R^n$.

\paragraph{Completing the anytime proof for possibly negative stepsizes.}
Let $(\widetilde h_k)_{k\ge1}$ be an infinite schedule.
We similarly define $P_k$, $M_k$, and $h_k$ by \eqref{eq:running-maximum-schedule} for all $k\ge1$, and set $H=(h_k)_{k\ge1}$.
For each $n$, apply \cref{lem:running-maximum-lifting} to the prefix $\widetilde H_n$.
Parts (i)--(iii) give the three lower bounds on $R_n(\widetilde H_n)$ required in the proof of \cref{lem:anytime-sequence-reduction}.
Their right-hand sides are exactly the same as those of the corresponding bounds on $R_n(H_n)$, with every quantity computed from $H_n$.
After these replacements, every subsequent step concerns only the nonnegative schedule $H$ and is unchanged.
This completes the proof of \cref{thm:anytime}.

\end{document}